\documentclass[a4paper, 11pt, UKenglish]{article}

\usepackage{babel}
\usepackage[utf8]{inputenc}
\usepackage[a4paper, margin=2.5cm]{geometry}
\usepackage[T1]{fontenc}

\usepackage{amsmath, amsthm, amssymb, mathrsfs}
\usepackage{enumitem}
\usepackage{authblk}
\usepackage{todonotes}
\usepackage{thmtools}
\usepackage{thm-restate}
\usepackage{mathtools}
\usepackage{hyperref}
\usepackage[capitalise]{cleveref}
\usepackage{xspace}
\usepackage{setspace}
\usepackage{tikz}
\usepackage{multicol}
\usepackage{aliascnt}
\usepackage{comment}

\usepackage{csquotes}
\usepackage[style=alphabetic,maxnames=999]{biblatex}
 \AtEveryBibitem{
 	\clearfield{editor}
 	\clearlist{editor}
 	\clearname{editor}
 	\clearfield{location}
 	\clearlist{location}
 	\clearname{location}
 	\clearfield{isbn}
 	\clearlist{isbn}
 	\clearname{isbn}
 	\clearfield{url}
 	\clearlist{url}
 	\clearname{url}
 	\clearfield{issn}
 	\clearlist{issn}
 	\clearname{issn}
 }

\declaretheorem[name=Theorem, numberwithin=section]{theorem}
\declaretheorem[name=Definition, style=definition, sibling=theorem]{definition}
\declaretheorem[name=Lemma, sibling=theorem]{lemma}

\declaretheorem[name=Corollary, sibling=theorem]{corollary}
\declaretheorem[name=Conjecture, sibling=theorem]{conjecture}
\declaretheorem[name=Claim, sibling=theorem]{claim}

\declaretheorem[name=Example, style=remark, sibling=theorem]{example}

\newenvironment{claimproof}[1][Proof.]{\begin{proof}[#1]}{\renewcommand{\qed}{\hfill$\blacksquare$}\end{proof}}
\newcommand{\claimqedhere}{\renewcommand{\qedsymbol}{\hfill$\blacksquare$}\qedhere}

\newcommand{\mainonly}[1]{}
\newcommand{\deferredsymbol}{$\star$}
\newcommand{\deferred}{\mainonly{\deferredsymbol}}

\crefname{claim}{Claim}{Claims}
\crefname{case}{case}{cases}
\crefname{prop}{Property}{Properties}

\newcommand{\set}[1]{\{#1\}}
\newcommand{\setof}[2]{\{#1\mid #2\}}

\let\origsubset\subset
\renewcommand{\subset}{\subseteq}
\newcommand{\subsett}{\origsubset}

\newcommand{\from}{\colon}
\newcommand{\eqdef}{\coloneqq}

\newcommand{\K}{\mathcal{K}}

\newcommand{\C}{\mathscr{C}}
\newcommand{\D}{\mathscr{D}}
\newcommand{\CC}{\C}
\newcommand{\DD}{\D}
\renewcommand{\P}{\mathcal{P}}
\newcommand{\Q}{\mathcal{Q}}

\newcommand{\Nn}{\mathcal{N}}
\newcommand{\F}{\mathcal{F}}

\renewcommand{\O}{\mathcal{O}}
\newcommand{\Oh}{\O}
\newcommand{\N}{\mathbb{N}}
\DeclareMathOperator{\mw}{mw}
\DeclareMathOperator{\dmw}{definable-mw}
\DeclareMathOperator{\tmw}{transient-mw}
\DeclareMathOperator{\pmw}{partition-mw}

\DeclareMathOperator{\shortmw}{short-mw}
\DeclareMathOperator{\tww}{tww}

\DeclareMathOperator{\scol}{scol}

\newcommand{\dist}{\mathrm{dist}}
\newcommand{\diam}{\mathrm{diam}}
\newcommand{\Ball}{\mathrm{Ball}}
\newcommand{\VC}{\mathrm{VCdim}}

\def\scol{\mathrm{scol}}
\def\wcol{\mathrm{wcol}}
\def\wreach{\mathrm{WReach}}
\def\sreach{\mathrm{SReach}}

\let\le\leqslant
\let\ge\geqslant
\let\leq\leqslant
\let\geq\geqslant

\let\qlt\prec
\let\qgt\succ
\let\qle\preceq

\let\emptyset\varnothing

\renewcommand{\cal}{\mathcal}

\let\preceq\preccurlyeq   % 

\title{Variants of Merge-Width and Applications\footnotemark}
\date{}

\usepackage{authblk}
\author[1]{Karolina Drabik}
\author[1]{Ma\"el Dumas}
\author[2]{Colin Geniet}
\author[1]{\\Jakub Nowakowski}
\author[1]{Micha{\l}~Pilipczuk}
\author[1]{Szymon Toru\'nczyk}

\affil[1]{University of Warsaw, Institute of Informatics, Poland}
\affil[2]{Institute for Basic Science (IBS), Discrete Mathematics Group, Daejeon, South Korea}

\begin{document}
\maketitle

\begin{abstract}
  Merge-width is a recently introduced family of graph parameters that 
  unifies treewidth, clique-width, twin-width, and generalised colouring numbers. 
We prove the equivalence of several alternative definitions of merge-width, thus demonstrating the robustness of the notion. Our characterisation via \emph{definable} merge-width uses vertex orderings inspired by generalised colouring numbers from sparsity theory, and enables us to obtain the first non-trivial approximation algorithm for merge-width, running in time $n^{\O(1)} \cdot 2^n$. We also obtain a new characterisation
of bounded clique-width in terms of vertex orderings, and establish that graphs of bounded merge-width admit sparse quotients with bounded strong colouring numbers, are quasi-isometric to graphs of bounded expansion, and admit neighbourhood covers with constant overlap. We also discuss several other variants of merge-width and  connections to adjacency labelling schemes.
\end{abstract}

\renewcommand{\thefootnote}{\fnsymbol{footnote}}
\footnotetext{\footnotemark
CG was supported by the Institute for Basic Science (IBS-R029-C1).  
KD, MD, JN, ST received funding from ERC grant BUKA (No.~101126229).
MD and MP were supported by the project BOBR that has received funding from ERC
under the European Union's Horizon 2020 research and innovation programme,
grant agreement No.~948057.}
\renewcommand{\thefootnote}{\arabic{footnote}}
\setcounter{footnote}{0}

\section{Introduction}

  Merge-width is a family of graph parameters introduced recently by Dreier and Toruńczyk \cite{merge-width}. The aim of merge-width is to unify and generalise  several well-known structural graph parameters such as tree-width, clique-width, twin-width, and generalised colouring numbers, which have been extensively studied due to their numerous applications in algorithmics, combinatorics, and logic. The model-checking problem for first-order logic is fixed-parameter tractable on graph classes of bounded merge-width \cite{merge-width}, assuming a suitable decomposition witnessing the boundedness is given. Those classes enjoy further good logical and combinatorial properties, such as closure under first-order transductions \cite{merge-width}, the strong Erdős--Hajnal property, and $\chi$-boundedness \cite{bonamy2025chiboundedness}. A central open question is whether merge-width can be efficiently approximated.

  % Roughly, merge-width is defined via a sequence of partitions of the vertex set of a graph, starting from the partition into singletons and ending with the partition with one part. At each step, some pairs of vertices are marked as \emph{resolved}, with the requirement that the unresolved pairs between any two parts of the current partition are either all edges or all non-edges. For $r\in\N$, the graph has \emph{radius-$r$ merge-width} at most $k$ if there exists such a sequence so that at any time, for every vertex $v$, at most $k$ parts can be reached from $v$ by a path length at most $r$ in the graph formed by the resolved pairs.

  In this paper, we explore the properties of merge-width further. We prove the equivalence of several alternative definitions of merge-width, demonstrating the robustness of the notion. 
  Among other findings, one of those characterisations 
   elucidates the connection of merge-width with central notions from sparsity theory \cite{sparsity}, by showing that merge-width can be characterized using vertex orderings similar to those used to define generalised colouring numbers. This potentially opens the avenue for applying methods used in sparsity theory, in the context of dense graphs.
   
   As an indication of the advantage of this characterisation, we use it to obtain a single-exponential time approximation algorithm for merge-width, running in time \mbox{$2^n\cdot n^{\Oh(1)}$} on $n$-vertex graphs. Another of our characterisations yields a certain normal form of merge sequences, which we use to show that graphs of bounded merge-width admit neighbourhood covers with constant overlap, and are quasi-isometric to graphs of bounded expansion. Before explaining our results in more details, we first recall the definition of merge-width.

  % \sz{possibly remove}In addition to these results, we also provide approximation results for classical optimisation problems such as the minimum dominating set and the maximum distance-2 independent set, in graphs of bounded radius-2 merge-width (given a suitable witness). These results further highlight the algorithmic potential of merge-width in solving combinatorial problems on graphs.

\subsection{Merge-width}
We will rely on the definition of merge-width expressed in terms of \emph{merge sequences} \cite[Sec.~3.1]{merge-width}. This variant is more flexible than the alternative definition in terms of \emph{construction sequences}
\cite[Sec.~1]{merge-width}, which in turn might be more intuitive.
See \cref{sec:mw-def} for a formal definition.
We also refer to \cite{merge-width} for illustrative examples.

A \emph{merge sequence} for a graph $G=(V,E)$ is a sequence 
$$(\P_1,R_1),(\P_2,R_2),\dots,(\P_m,R_m),$$
where $\cal P_1\prec\dots\prec\cal P_m$ is a sequence
of ever coarser partitions of the vertex set $V$, with $\P_1$ the partition into singletons and $\P_m$ the partition with one part $V$, and $R_1\subseteq \dots\subseteq R_m\subseteq \binom{V}{2}$ are sets of \emph{resolved} pairs of vertices. The requirement is that for every $t$ and every pair of parts $A,B\in \cal P_t$ (possibly $A=B$), the \emph{unresolved} vertex pairs  between $A$ and $B$ are either all edges or all non-edges in $G$:
$$AB\setminus R_t \subseteq E \qquad\textit{or}\qquad AB\setminus R_t \subseteq  \tbinom{V}{2}\setminus E\qquad\textit{for all $A,B\in \P_t$.}$$
% (Thus, $A$ and $B$ have a \emph{default connection} --- adjacency or non-adjacency --- which applies to all vertex pairs in $AB\in R_t$.)
For $r\in\N$, the \emph{radius-$r$ width} of a merge sequence is the maximum, over all 
$t\in [m-1]$ and $v\in V$, of the number of parts in $\cal P_t$ that are at distance at most $r$ from $v$, where distances are measured in the graph $(V,R_{t+1})$ formed by the resolved pairs in the subsequent\footnote{This offset 
is essential in the case when parts of $\P_{t+1}$ are unions of many parts of~$\P_t$.
However, in the definition of merge-width, one can without loss of generality consider sequences in which $\P_{t+1}$ is obtained from $\P_t$ by merging exactly two parts. In this case, distances can be measured in the graph $(V,R_t)$, resulting in the same merge-width, $\pm 1$.} step~$t+1$.
The \emph{radius-$r$ merge-width} of a graph $G$, denoted by~$\mw_r(G)$, is the least radius-$r$ width of a merge sequence for $G$. 
A graph class~$\C$ has \emph{bounded merge-width} if $\mw_r(\C)<\infty$ holds\footnote{For a graph parameter $f$ and graph class $\C$, we denote $f(\CC)=\sup \{f(G): G\in\C\}$.} for every $r\in\N$. See \cref{ex:run-ex} for a simple example.

Examples of graph classes of bounded merge-width include:
sparse graph classes, such as the class of planar graphs, every class that excludes some graph as a (topological) minor, and more generally, 
every class of bounded expansion (a notion from sparsity, \cite{sparsity});
dense graph classes, such as the class of unit interval graphs, or every proper hereditary class of permutation graphs, classes of bounded clique-width, and more generally, classes of bounded twin-width \cite{tww1}.

%The \emph{radius-$\infty$ merge-width of $G$} is defined 
%as above, but we count the number of parts of $\P_t$ at any finite distance from $v$ in the graph $(V,R_{t+1})$. 
%% Equivalently, 
%% $\mw_\infty(G)=\mw_r(G)$, where $r$ is any number with $r\ge \left|V(G)\right|-1$. 
%It is known \cite{merge-width} that $\mw_\infty(G)$ is functionally equivalent to the clique-width of $G$. In particular, for every class $\C$ of bounded treewidth, or of bounded clique-width, we have $\mw_\infty(\C)<\infty$.
%

\subsection{Main result}
We give several equivalent characterisations of merge-width, using alternative parameters discussed below.
Intuitively, our main result states that the full data of a merge sequence can be approximately recovered from some partial information,
such as a sequence of partitions, or a suitably chosen vertex ordering.
This provides more light-weight variants of merge-width which are sometimes better behaved, leading to new applications of the notion.

\begin{theorem}\label{thm:intro}
  The following are equivalent for a graph class $\C$:
  % \begin{multicols}{2}
    \begin{enumerate}[nosep]
      \item\label{item:merge-width} $\C$ has bounded merge-width,
      \item\label{item:transient-width} $\C$ has bounded transient merge-width,
      \item\label{item:partition-width} $\C$ has bounded partition merge-width,
      \item\label{item:definable-width} $\C$ has bounded definable merge-width.
    \end{enumerate}
  % \end{multicols}
  % Moreover, all translations between 
%  these forms of merge-width 
%  are computable in polynomial time.
\end{theorem}
Each item in the theorem corresponds to a family of parameters, indexed by radius~$r$, which is defined by modifying the definition of merge-width. The modifications are explained briefly below, and in greater detail in \Cref{sec:variants}.

% A typical implication in the theorem, \eqref{item:merge-width}$\rightarrow$\eqref{item:definable-width} proceeds by showing that the radius-$r$ definable merge-width is bounded in terms of the radius-$(17r+3)$ merge-width, etc.

\begin{description}[leftmargin=0pt,labelindent=0pt]
  \item[Transient merge-width] is defined similarly as merge-width, but the monotonicity condition $R_1\subseteq R_2\subseteq\ldots\subseteq R_m$ is dropped. This is equivalent to specifying a sequence 
  $$(\cal P_1, G_1), (\cal P_2, G_2), \ldots, (\cal P_m,G_m)$$
  with $\cal P_1\prec\dots\prec\cal P_m$ as before,
  where each graph $G_t$ is a \emph{$\cal P_t$-flip} of~$G$ --- obtained from $G$ by complementing the adjacency between some pairs of parts in $\cal P_t$ --- and the radius-$r$ width is measured by considering distances in $G_{t+1}$ rather than in $(V,R_{t+1})$. 
  \item[Partition merge-width] is defined by providing only a sequence of ever coarser partitions $\cal P_1\prec \cdots \prec \cal P_m$. Distances are measured with respect to the \emph{flip-metric} of $\cal P_{t+1}$, defined for each pair of vertices as their maximum distance in any $\cal P_{t+1}$-flip $G'$ of~$G$.
  \item[Definable merge-width] is specified by an enumeration (or total order) of $V(G)$.
  For each $i\in[n]$ consider the partition $\cal P_{i}$ of $V(G)$,
  which partitions the last $i$ vertices
  according to their neighbourhood in the set of the $n-i$ remaining vertices, which in turn is partitioned into singletons. 
  This yields a sequence of partitions $\cal P_{1}\prec\dots\prec \cal P_{n}$;
  distances are measured with respect to the flip-metric of $\cal P_t$.
\end{description}

Transient merge-width is a very natural weakening of merge-width which is still useful combinatorially. For instance, in the results of Bonamy and Geniet \cite{bonamy2025chiboundedness} --- proving the Strong Erdős--Hajnal property and 
linear neighbourhood complexity for classes of bounded merge-width --- it is enough to assume a bound on the transient merge-width of the graphs.
On the other hand, for the  purposes studied in \cite{merge-width} --- first-order model checking, closure under transductions, and the relationship with flip-width --- the original definition of merge-width appears much more useful. Note that the equivalence of bounded merge-width and bounded transient merge-width proved in \Cref{thm:intro} does not respect the radii: while we clearly have $\tmw_1(G)\le \mw_1(G)$, in the other direction our proof only bounds $\mw_1(G)$ in terms of $\tmw_{7}(G)$.

Definable merge-width is inspired by, and generalises 
the \emph{generalised colouring numbers} (\emph{weak} and \emph{strong} colouring numbers) studied in sparsity theory, which are also defined through a vertex ordering.
In particular, the construction in \cite[Lem.~7.5]{merge-width} shows that the definable merge-width with radius~$r$ of a graph $G$ is bounded in terms of the $(r+1)$-weak colouring number of $G$.
Each of the partitions $\P_i$ considered in the definition of definable merge-width is definable (as an equivalence relation) using a first-order formula $\pi_i(x,y)$, which uses a unary predicate marking the last $i$ vertices in the enumeration; hence the term \emph{definable}. In particular, the binary structure $G_i$ consisting of~$G$ and the partition $\P_i$ (again as an equivalence relation)
can be obtained from $G$ using a fixed \emph{first-order transduction} (see e.g.\ \cite{lsd} for a definition). This implies that $G_i$ is ``tame'' if $G$ is, e.g. its clique-width, twin-width or merge-width can be bounded in terms of the corresponding parameters of $G$.
This is in contrast with the other considered variants of merge-width.
\medskip

The proof of \cref{thm:intro} relies on two \emph{metric conversion} lemmas, \cite[Lemma~8]{flip-separability}, and \cite[Theorem~3.5]{local-cliquewidth}.
Informally, each of the previous variants of merge-width uses some notion of distance defined by a vertex partition~$\P$.
These lemmas show how these different distances can be approximated by each other, at the cost of a controlled modification of~$\P$.
To apply these results in the setting of merge-width, there are two difficulties to overcome.
First, these metric conversion lemmas were previously considered for partitions~$\P$ with a bounded number of parts;
for merge-width, we instead assume that the number of parts is bounded only in the neighbourhood of each vertex.
Secondly, when applying them to a partition sequence $\P_1 \qlt \dots \qlt \P_m$,
we need to ensure that as a result, we obtain a sequence of ever coarser partitions. The application of metric conversion lemmas to a sequence of partition is inspired by~\cite{sparseflip}, where a strong colouring order is constructed from a (transient) merge sequence in graphs excluding some biclique $K_{t,t}$.

\medskip
Dreier and Toruńczyk~\cite{merge-width} also define merge-width at radius~$\infty$, considering connectivity rather than distances. It is functionally equivalent to clique-width.
The previous variants of merge-width can also be defined at radius-$\infty$, and the equivalence still holds:
\begin{theorem}\label{thm:cw}
  The following are equivalent for a graph class~$\C$:
  \begin{enumerate}[nosep]
    \item\label{item:cw-cw} $\C$ has bounded clique-width,
      \item\label{item:cw-merge-width} $\mw_\infty(\C)<\infty$,
      \item\label{item:cw-transient-width} $\tmw_\infty(\C)<\infty$,
      \item\label{item:cw-partition-width} $\pmw_\infty(\C)<\infty$,
      \item\label{item:cw-definable-width} $\dmw_\infty(\C)<\infty$.
    \end{enumerate}
\end{theorem}

In a different direction, \cite{merge-width} considers classes~$\C$ with \emph{almost bounded} merge-width,
meaning that for any $r \in \mathbb{N}$ and $\varepsilon > 0$, for sufficiently large~$n$ and for any graph~$G \in \C$ with~$n$ vertices, $\mw_r(G) \le n^\varepsilon$.
Such classes include for instance all \emph{nowhere dense} graph classes considered in sparsity theory.
Every hereditary graph class of 
almost bounded merge-width is \emph{monadically dependent} --- a notion which has gained recent attention in finite model theory and structural graph theory, in connection with the model-checking problem for first-order logic \cite{flip-breakability,demmpt,merge-width}.
Conversely, it is conjectured that all monadically dependent graph classes have almost bounded merge-width \cite[Conjecture~1.18]{merge-width}, and that for hereditary classes, monadic dependence is equivalent to the class having a tractable model checking problem for first-order logic.

Almost bounded transient, partition, and definable merge-width can be defined similarly as almost bounded merge-width.
Our main results in fact show that the different variants of merge-width are bounded by polynomial functions of each other, assuming a fixed bound on the VC-dimension of the considered graphs.
It follows that the `almost bounded' conditions are also equivalent
(as each implies a bound on the VC-dimension, see \Cref{cor:almost-bounded-VC}):
\begin{theorem}\label{thm:almost-bounded}
 The following are equivalent for a hereditary graph class~$\C$:
 \begin{enumerate}[nosep]
     \item\label{item:almost-merge-width} $\C$ has almost bounded merge-width,
     \item\label{item:almost-transient-width} $\C$ has almost bounded transient merge-width,
     \item\label{item:almost-partition-width} $\C$ has almost bounded partition merge-width,
     \item\label{item:almost-definable-width} $\C$ has almost bounded definable merge-width.
   \end{enumerate}
\end{theorem}

\subsection{Applications}
We now list several applications of \Cref{thm:intro}.

\paragraph{Exponential-time approximation of merge-width}
A total order on the vertex set as in the definition of definable merge-width can be approximated 
as follows. Given an $n$-vertex graph, for each set $S \subset V(G)$ we compute the partition $\cal N_S$
into atomic types over $S$ (that is, where each element of~$S$ forms a singleton part, and the vertices of $V\setminus S$ are partitioned according to their neighbourhood in $S$), and then we approximate the radius-$r$ width
of this partition in polynomial time using a lemma from~\cite{flip-separability}.
As there are $2^n$ many subsets $S\subset V(G)$, this takes time $2^n\cdot n^{\O(1)}$ in total. We can then find a chain of sets $S_0\subsett S_1\subsett \dots\subsett S_n$ of length $n+1$
minimizing the estimated parameter using dynamic programming, thus obtaining a total order $<$ on $V(G)$.
This leads to the following\footnote{We thank Václav Blažej for suggesting this result.}.
\begin{restatable}{theorem}{apxMW}\label{thm:approx-intro}
  There is an algorithm that, given~$G$ an $n$-vertex graph and $r\in \N$,
  computes in time $n^{\O(1)}\cdot 2^n$ a merge sequence for~$G$ with radius-$r$ width at most $f(\mw_{734r}(G))$, where $f(\ell) \le \ell^{\O(\ell)}$.
\end{restatable}
Note that for the related, but more restrictive parameter twin-width, whose efficient approximation is also a major open problem, it is currently unknown whether there exists any approximation algorithm with running time $2^{\O(n)}$ whose approximation factor would depend only on the value of the optimum. Therefore, \cref{thm:approx-intro} suggests that merge-width might be easier to approximate than twin-width while providing a similar strength with respect to e.g. the model-checking problem for first-order logic.

In \cref{thm:approx-intro}, we did not make an effort to optimise the upper bound on the merge-width of the obtained merge sequence.

\paragraph{Sparse quotients, quasi-isometries, and neighbourhood covers}
As a result of the translation from partition merge-width to merge-width
(implication \eqref{item:partition-width}$\rightarrow$\eqref{item:merge-width} in \Cref{thm:main1}),
we obtain merge sequences with an additional property: for any pair
$uv$ that is resolved at the end of the merge sequence (i.e. $uv\in R_m$), $u$ and $v$ are at distance at most $6$ in $G$.
In \Cref{sec:quasi-iso} we use this property
to construct a partition $\P$ of $V(G)$ such that the quotient 
graph $G/\P$ has bounded generalised colouring numbers, and each part of $\P$ has (weak) diameter at most $12$ in~$G$. 

This has two (related) consequences.
In \cref{thm:quasi-isom} we show that
every graph $G$ of bounded radius-$(18r+1)$ merge-width is \emph{quasi-isometric} to a graph $G'$ whose $r$-strong colouring number is bounded.  Intuitively, this means that 
  graphs of bounded merge-width have the same large-scale geometry as graphs of bounded expansion.
  % Our result also applies to other graph properties which are closed under transductions and imply bounded merge-width (such as bounded twin-width), and their sparse counterparts (such as bounded sparse-twin-width).
  % $G$ and $G'$ have roughly equal distances, specifically,
  % $\dist_G(u,v)=\Theta(\dist_{G'}(u,v))$ for all vertex pairs $u,v$. 
  Results of this form have recently attracted attention
\cite{nguyen2025asymptoticstructureicoarse,hickingbotham2025graphsquasiisometricgraphsbounded,distel2025improvedquasiisometrygraphsbounded}, as part of the recent coarse graph theory programme~\cite{coarse-graphs}.
  
In \cref{thm:nbhdCov} we show that every graph $G$ of bounded radius\nobreakdashes-$37$ merge-width admits 
  a \emph{neighbourhood cover} of bounded radius and bounded overlap.
  Those are families of sets which have similar properties as balls of bounded radius in $G$, but are more manageable algorithmically in some contexts.
Neighbourhood covers originate in the study of distributed network algorithms, and have found a key application in the model-checking results for nowhere dense classes \cite{gks} and monadically stable graph classes \cite{dms,demmpt},
and results allowing to \emph{sparsify} monadically stable graph classes \cite{quasi-bushes,horizons-stable,dreier2026efficientreversaltransductionssparse}, that is, encode graphs from those classes in sparse graphs, in a first-order definable way.

\subsection{Other variants of merge-width}
In addition to the notions of merge-width discussed above, in \Cref{sec:other} we also consider the following variants: 
\begin{description}
\item[short merge-width,] where (i) the merge sequence has length logarithmic in the size of the graph and (ii) every part of every partition in the merge sequence consists of a constant number of parts of the partition from the previous time step;
\item[universal merge-width,] which requires the existence of a single merge sequence whose radius-$r$ width is bounded by a function of $r$, for all $r\in \N$; and 
\item[positive merge-width,] where only edges of $G$ can be resolved, i.e.~$R_m\subseteq E(G)$ (equivalently, $R_t\subset E(G)$ for all $t\in[m]$).
\end{description}

\emph{Short merge-width} is inspired by a similar notion of \emph{short contraction sequences} studied in the context of twin-width \cite{tww2-journal}.
Clearly, classes of bounded short merge-width have bounded merge-width.
We conjecture the converse implication (see \cref{conj:short}),
and observe that classes of bounded twin-width and classes of bounded expansion 
have bounded short merge-width, and that this property is preserved under first-order transductions (for hereditary classes, see \cref{lem:short-transduction}).
Generalising \cite[Thm.~2.8]{tww2-journal}, short merge-width allows to obtain adjacency labelling schemes with short labels: that is, for every $n\in\N$, each $n$-vertex graph $G$ with $\shortmw_1(G)$ bounded by a constant can be vertex-labelled using bit-strings of length $\O(\log n)$, so that the adjacency between two vertices can be determined only by inspecting their labels. 

% \begin{conjecture}
%   Every graph class $\C$ of bounded merge-width has bounded short merge-width: $\shortmw_r(\C)<\infty$  for every $r\in\N$.
% \end{conjecture}
% \sz[inline]{
% This conjecture holds in special cases: 
% for classes of bounded expansion, and for classes of bounded twin-width.
% Furthermore, it follows from the proof in \cite{merge-width} that 
% classes of bounded short merge-width are closed under first-order transductions.
% In particular, this implies that classes of structurally bounded expansion (that is, first-order transductions of classes of bounded expansion) also have bounded short merge-width. 
% }
\begin{restatable}{theorem}{thmAdjLab}\label{thm:adj-labelling}
Every class of graphs of bounded radius-$1$ short merge-width admits an adjacency labelling scheme of length $\Oh(\log n)$.
%Every graph $G$ admits an adjacency labelling scheme with labels of size $\O_k(\log n)$, where $n=|V(G)|$ and $k=\shortmw_1(G)$.
\end{restatable}
As mentioned, we conjecture that all classes of bounded merge-width have bounded short merge-width. If true, with \Cref{thm:adj-labelling} this would imply that they also have $\O(\log n)$-adjacency labelling schemes, confirming a conjecture of Dreier and Toruńczyk \cite{merge-width}. As discussed in \cite{merge-width}, this could be a route towards proving the \emph{Small Implicit Graph Conjecture} of \cite{bonnet_et_al:LIPIcs.ITCS.2025.21}.
\medskip

A graph class $\C$ has \emph{universally bounded merge-width} 
if there is a function $f\from\N\to\N$ such that each graph $G\in\C$ has a merge sequence whose radius-$r$ width is at most $f(r)$ for all $r\in\N$. (Thus, a single merge sequence is used as a witness to small merge-width for all~$r$ simultaneously.)
Such classes include classes of bounded twin-width and of bounded expansion, and are closed under first-order transductions.
We conjecture that all classes of bounded {merge-width} have universally bounded merge-width (see \cref{conj:univ}). For classes of universally bounded merge-width, the previously mentioned \cref{thm:quasi-isom} can be strengthened to obtain the following:%
\begin{restatable}{theorem}{quasiIsoBE}\label{thm:quasi-iso-be}
Let $\C$ be a class with universally bounded merge-width. Then there is a class of bounded expansion $\D$ such that every $G\in\C$ is $(13,12)$-quasi-isometric to some $H\in\D$.
\end{restatable}

Finally, classes of bounded \emph{positive merge-width} --- where 
only edges can be resolved --- include classes of bounded twin-width or of bounded expansion.
For such classes, the results from \Cref{sec:quasi-iso} (sparse quotients, quasi-isometries, and neighbourhood covers) are much easier to obtain (see \cref{sec:tww-app}).
However, not every class of bounded merge-width has bounded positive merge-width:  we prove in \Cref{lem:co-cubic} that the class of edge-complements of cubic graphs has
unbounded positive merge-width, using the fact that it has unbounded twin-width~\cite{tww2-journal}.

% \subsection{Approximation algorithms for classical problems}
% \sz{write or remove}

\section{Preliminaries}\label{sec:prelims}
%\emph{The proofs of results marked with (\deferredsymbol) can be found in the appendix.}
\renewcommand{\mainonly}[1]{#1}

\subsection{Basic notation and graphs}
For $n\in \N$, we write $[n]\eqdef\set{1,2,\ldots,n}$.
For two sets $A,B$, we write $AB \eqdef\setof{ab}{a\in A,b\in B,a\neq b}$, where $ab$ denotes the unordered pair $\{a,b\}$, and $\binom{A}{2} \eqdef AA$.
Graphs are simple, undirected, and finite, that is, a graph $G$ consists of a finite set $V(G)$ of vertices and a set $E(G)\subseteq \binom{V(G)}{2}$ of edges.
The edge-complement of~$G$ is~$\overline{G}$, with $V(\overline{G})\eqdef V(G)$ and $E(\overline{G}) \eqdef \binom{V(G)}{2} \setminus E(G)$.
By $\dist_G(x,y)$ we denote the length of a shortest path between $x$ and~$y$ in $G$ (or $\infty$ if no path exists), and denote $\Ball^r_G(v) \eqdef\setof{w\in V(G)}{\dist_G(v,w)\le r}$. By $N_G(v)$ we denote the open neighbourhood of~$v$ in $G$, that is, the set of vertices adjacent to $v$; we may omit the subscript if $G$ is clear from the context. We write $\overline{N_G}(v)\eqdef N_{\overline{G}}(v)$.
Two sets $A,B\subseteq V(G)$ are \emph{complete} if $AB\subseteq E(G)$, and \emph{anti-complete} if $AB\cap E(G)=\emptyset$,
and are \emph{homogeneous} if they are either complete or anti-complete.

\subsection{VC-dimension and neighbourhood complexity}
% If~$A$ is a subset of vertices in a graph~$G$, the set of \emph{neighbourhood types over~$A$} is
% \[ N^A \eqdef \{N(x) \cap A : x \not\in A\}. \]
A set~$A$ of vertices of a graph $G$ is \emph{shattered} 
if $\setof{N(v)\cap A}{v\in V(G)}=2^A$.
% that is, $N^A=2^A$.
The \emph{VC-dimension} of~$G$ is the maximum size of a shattered subset of $V(G)$.
The \emph{neighbourhood complexity function} (or \emph{shatter function}) is defined by:
\[ \pi_G(m) \eqdef \max_{\substack{A\subseteq V(G),\\|A|\le m}}\Big|\setof{N(v)\cap A}{v\in V(G)}\Big|.\]
This extends to a class~$\C$ of graphs as $\pi_\C(m) \eqdef \max_{G \in \C} \pi_G(m)$.

Note the trivial bound $\pi_G(m) \le 2^m$. The fundamental Sauer-Shelah-Perles lemma \cite{sauer72,shelah72} states that this bound is polynomial in graphs of bounded VC-dimension.
%The Sauer-Shelah lemma states that when~$G$ has VC-dimension~$d$, this function is a polynomial $\pi_G(k) \le \O(k^d)$, where $\O$ hides a universal constant (independent of $G$ and $d$).

\begin{lemma}[Sauer-Shelah-Perles lemma] \label{sauer_shelah_lemma}
	Let $G$ be a graph of VC-dimension~$d$. Then $$\pi_G(m) \le \O(m^d)\qquad\text{for all $m\in\N$.}$$
\end{lemma}
For graphs of bounded radius-$2$ merge-width
(see \Cref{sec:mw-def} for a definition), 
Bonamy and Geniet \cite[Thm.~1.5]{bonamy2025chiboundedness}
proved that the neighbourhood complexity function is linear.

\begin{theorem}\label{thm:nbd_complexity}
	For any graph $G$ with $\mw_2(G)= k$ we have:
	$$\pi_G(m)\leq k2^{k+2}\cdot m\qquad\text{for all $m\in\N$.}$$
\end{theorem}

\subsection{Partitions}
\label{sec:partitions}
A \emph{transversal} of a partition $\cal P$ of $V$ is a set $S\subseteq V$ 
such that $|S\cap P|=1$ for all $P\in\cal P$.
Given two partitions~$\P,\Q$ of $V$,
we write 
$\P\preceq\Q$ and say that $\P$ \emph{refines} $\Q$, or $\Q$ is \emph{coarser} than $\P$, if every part of $\P$ is contained in a part of $\Q$. 
We write $\P\prec \Q$ if additionally $\P\neq \Q$. 
A \emph{chain of partitions} of a set $V$ is
a sequence 
$$\cal P_1\prec \cal P_2\prec \dots\prec \cal P_m$$ of distinct partitions of $V$.
A \emph{maximal chain of partitions} is such a sequence of length $n=|V|$, 
equivalently, $\cal P_1$ is the partition into singletons, $\cal P_n$ has one part, and for every $1\le i<n$, $\cal P_{i+1}$ is obtained from $\cal P_i$ by merging exactly two parts. 

The common refinement of $\P$ and $\Q$ is the partition
\[ \P \wedge \Q \eqdef \{P \cap Q : P \in \P, Q \in \Q\}\setminus\set{\emptyset}. \]
For a partition $\P$ of $V$ and set~$S \subset V$, the \emph{quotient set}  is defined as 
\[S/\P \eqdef \{P\in \P\mid P\cap S\neq \emptyset\}.\]
% while by $S/\P$ we denote 
% the quotient set, whose elements are  equivalence classes of $\P|_S$.
% \sz{check if $\P|_{S}$  and $S/\P$ are consistently used. Is $\P|_{S}$ used at all? If not, remove.}

Given a set~$S \subseteq V(G)$ of vertices, we denote by~$\Nn_S$ the partition of $V(G)$ into \emph{atomic types} over $S$, which partitions $S$ into singletons, and two vertices $u,v \notin S$ belong to the same part of~$\Nn_S$
if and only if $N(u) \cap S = N(v) \cap S$.

\subsection{Flips and flip metric}\label{sec:flips}
For a graph $G$ and a partition $\P$ of $V(G)$, a \emph{$\P$-flip} of $G$ is any graph~$G'$ with vertex set $V(G)$ such that for any parts $X,Y \in \P$ (possibly $X=Y$) we have either $XY\cap E(G')=XY\cap E(G)$ or $XY\cap E(G')=XY\setminus E(G)$. In the latter case, we say that the pair $X,Y$ is \emph{flipped} in $G'$ with respect to $G$. For a set $S\subset V(G)$,
an \emph{$S$-flip} of $G$ is an  $\Nn_S$-flip of $G$, where $\Nn_S$ is the partition into atomic types.

In a graph~$G$, the \emph{flip metric} associated with the partition~$\P$ of $V(G)$ is defined, following \cite{flippers}, as 
\[ \dist_{\P}(x,y) = \max_{\text{$G'$ $\P$-flip of $G$}} \dist_{G'}(x,y), \]
where the maximum ranges over all 
$\cal P$-flips $G'$ of $G$.
For $v\in V(G)$ and $r\in\N$, define \[\Ball_{\P}^{r}(v)\eqdef\{x\in V(G)\colon\dist_{\P}(x,v)\le r\}=\, \bigcap_{\text{$G'$ $\P$-flip of $G$}} \Ball_{G'}^{r}(v). \]
The following is immediate from the definition.
\begin{lemma}\label{lem:flip-metric-refine}
  If~$\P$ refines~$\Q$ then for any vertices $x,y$,
  \[ \dist_\P(x,y) \ge \dist_{\Q}(x,y)\quad \text{and}\quad \Ball^r_\P(x)\subseteq \Ball^r_\Q(x).\]
\end{lemma}

\subsection{Generalised colouring numbers}
  Fix a graph $G$, number $r\in\mathbb{N}$, and total order $<$ on $V(G)$. 
  A vertex $u$ is \emph{weakly $r$-reachable} from a vertex $v \geq u$ if  there is a path of length at most $r$ from $v$ 
  to $u$ such that $w> u$ for all internal vertices $w$ of that path.
  Similarly, $u$ is \emph{strongly $r$-reachable} from a vertex $v \geq u$ if there is a path of length at most~$r$ from~$v$ to~$u$ such that $w> v$ for all internal vertices $w$ of that path.
  
  Let $\wreach_r^<(v)$ be the set of vertices that are weakly $r$-reachable from $v$ with respect to the order $<$. Denote
  \[
  \wcol_r^<(G) \coloneqq \max_{v \in V(G)} |\wreach_r^<(v)|
  \]
  and set $\wcol_r(G)$ to be the minimal value of~$\wcol_r^<(G)$ across all the total orders on~$V(G)$. 
  The notions of $\sreach_r^<(G)$, $\scol_r^<(G)$ and $\scol_r(G)$ involving strong $r$-reachability are defined analogously.
  \begin{lemma}[\cite{gen-col-nb}]
      \label{lem:col-nb-equiv}
      Fix a graph~$G$, total order~$<$ on $V(G)$, and $r\in\mathbb{N}$. Then
      \[
      \scol_r^<(G)\ \leq\ \wcol_r^<(G)\ \leq\ (\scol_r^<(G))^r.
      \]
  \end{lemma}
  By a characterisation of \cite{zhu},
   a graph class $\C$ has \emph{bounded expansion} if and only if $\scol_r(\C)<\infty$ for every $r\in\N$
  (equivalently, $\wcol_r(\C)<\infty$ for every $r\in\N$).

\subsection{Twin-width}\label{sec:tww-prelims}
The \emph{twin-width} of a graph $G$, denoted by $\tww(G)$, is the minimum integer $d$ such that there exists a maximal chain 
$$\P_1\prec \P_2\prec \dots\prec \P_n$$
of partitions of $V(G)$ (called a \emph{contraction sequence} in this context) such that 
for every $t\in[n]$, every part $A\in\P_t$ is non-homogeneous to at most $d$ parts $B\in \P_t$ with $B\neq A$.

\subsection{Merge-width}\label{sec:mw-def}
Let~$\P$ be a partition of the vertices $V$ of a graph~$G$, and $R \subseteq \binom{V}{2}$ be a set of pairs of vertices of~$G$.
By $\dist_R(x,y)$ and $\Ball_R^r(x)$ we mean the corresponding notions in the graph $(V,R)$.
The \emph{radius-$r$ width} of $(\P,R)$~is
\[ \max_{v \in V} \left|\Ball^r_R(v)/\P\right|. \]
We say that~$\P$ is \emph{homogeneous modulo $R$} (in $G$) if for any parts $X,Y \in \P$ (possibly $X=Y$)
and pairs $xy,x'y'\in XY\setminus R$, we have that $xy \in E(G)$ if and only if $x'y' \in E(G)$.
Equivalently, there is a $\P$-flip~$G'$ of~$G$ such that $E(G')\subset R$.

\begin{lemma}
  \label{lem:resolved-flip-metric}
  If~$\P$ is homogeneous modulo $R$ in~$G$, then for all~$x,y \in V(G)$ we have
  \[ \dist_R(x,y) \le \dist_\P(x,y). \]
\end{lemma}
\begin{proof}
  Let $G'$ be a $\P$-flip of $G$ such that $R \supseteq E(G')$.
  Then:
  \[ \dist_R(x,y) \le \dist_{G'}(x,y) \le \dist_{\P}(x,y). \qedhere \]
\end{proof}

\begin{definition}[{\cite[Definition 3.1]{merge-width}}]
A \emph{merge sequence} for a graph~$G$ is a sequence $$(\P_1,R_1),\dots,(\P_m,R_m)$$ such that:
\begin{enumerate}
  \item $\P_1\preccurlyeq \P_2\preccurlyeq\ldots\preccurlyeq \P_m$ is a sequence of ever coarser partitions of $V(G)$ with~$\P_1$ the partition into singletons and $\P_m$ the partition with one part, 
  \item $R_1 \subseteq \dots \subseteq R_m \subseteq \binom{V(G)}{2}$ is a monotone sequence of pairs of vertices, and
  \item $\P_t$ is homogeneous modulo $R_t$, for $t=1,\ldots,m$.
\end{enumerate}
The \emph{radius-$r$ width} of this merge sequence is the maximum radius-$r$ width of $(\P_{t},R_{t+1})$,
for  $1 \le t < m$.
Finally, the \emph{radius-$r$ merge-width} of~$G$, denoted by~$\mw_r(G)$, is the minimum radius-$r$ width of a merge sequence for~$G$.
\end{definition}
Note that the mismatched indices in $(\P_t,R_{t+1})$ are intentional, and forbid one from merging many parts and adding many resolved pairs all at once when going from step~$t$ to~$t+1$.

In the definition of merge-width, the sequence of partitions $\P_1\preceq\ldots\preceq \P_m$ may be equivalently required to be a {maximal chain of partitions} of $V(G)$.
Indeed, given any merge sequence $(\P_1,R_1),\dots,(\P_m,R_m)$,
one can transform it into a merge sequence whose sequence of partitions is a maximal chain of partitions as follows: if $\P_i=\P_{i+1}$ for some $1\le i<m$, then we can drop the pair $(\P_{i+1},R_{i+1})$ in the sequence, and if $\P_i\prec \P\prec \P_{i+1}$, then we may insert the pair $(\P,R_{i+1})$ into the sequence. Those operations do not increase the radius-$r$ width of the merge sequence.

Toruńczyk~{\cite[Thm.~5.24]{flip-width}} showed that the VC-dimension of a graph $G$ is linear in its radius-$1$ flip-width. 
The proof can be easily adapted for radius-$1$ merge-width (see also~\cite[Lemma~7.20]{merge-width}).
\begin{theorem}\label{thm:vc_mw1}
	Any graph $G$ satisfies $\VC(G) \le \O(\mw_1(G))$.
\end{theorem}

To illustrate the notion of merge-width and its variants given in the next section, we present a particular class of graphs with small width parameters. 

\begin{example}
	\label{ex:run-ex}
	For a fixed $m \in \N$, consider the (perfect) binary tree~$T_m$ of depth~$m$, that is, the root has depth~$0$ and the~$2^m$ leaves have depth~$m$. 
	Let $C_m$, the $2^m$-vertex \emph{universal cograph}, be the graph whose vertices are the leaves of~$T_m$, and two different leaves are adjacent in $C^m$ if and only if their lowest common ancestor in~$T_m$ has odd depth (see \cref{fig:cograph_partition}).
	
	We construct a merge sequence $(\P_1,R_1), \ldots, (\P_{m+1},R_{m+1})$ for~$C_m$. 
	For any $w\in V(T_m)$, let~$P_w$ be the leaves in~$T_m$ below (or equal to)~$w$. 
	For $i\leq m+1$, define 
	\[\P_i=\{P_w\mid w\in V(T_m)\text{ has depth $m+1-i$}\}.\]
	
Note that whenever two different $v,w\in V(T_m)$ are at the same depth, the sets $P_v,P_w$ are homogeneous. Moreover, at any time $i>1$, each part of~$\P_i$ is a union of exactly two parts of~$\P_{i-1}$. Therefore, taking $R_i= \bigcup_{P\in\P_{i-1}}\binom{P}{2}$, the partition $\P_i$ is homogeneous modulo~$R_i$. Clearly, $R_i\supset R_{i-1}$.
	
	In the construction above, for any~$i\leq m$, no two parts of~$\P_i$ are connected by an edge of~$R_{i+1}$, so that $\mw_r(C_m)=1$ for every $r$. \qed
\end{example}

\begin{figure}
	\centering
	\includegraphics[width=0.5\textwidth]{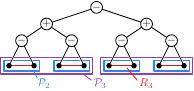}
	\caption{
  % color codes purple : 0.569 0.255 0.675
% blue : 0.208 0.518 0.894
% red : 0.878 0.106 0.141
% make the colors below match the color codes a
  The tree $T_3$ of \cref{ex:run-ex}, with partition \textcolor[rgb]{0.208, 0.518, 0.894}{$\P_2$} in \textcolor[rgb]{0.208, 0.518, 0.894}{blue}, \textcolor[rgb]{0.569, 0.255, 0.675}{$\P_3$} in \textcolor[rgb]{0.569, 0.255, 0.675}{purple} and set of resolved pairs \textcolor[rgb]{0.878, 0.106, 0.141}{$R_3$} in \textcolor[rgb]{0.878, 0.106, 0.141}{red}.}
	\label{fig:cograph_partition}
\end{figure}
\section{Equivalent Variants of Merge-width}\label{sec:variants}
In this section, we define the main variants of merge-width studied
in this paper, and state the obvious relationships between them.
\Cref{thm:intro} states that the variants presented in this section are all equivalent to the usual notion of merge-width.

\subsection{Transient merge-width}
\begin{definition}  Fix $r\in\N$.
  The \emph{radius-$r$ transient merge-width} of a graph $G$, denoted by $\tmw_r(G)$, is the minimum number $k$ such that there is a sequence  $$(\P_1,G_1),\dots,(\P_m,G_m)$$
    with the following properties:
    \begin{itemize}
      \item $\P_1\preccurlyeq\P_2\preccurlyeq\ldots\preccurlyeq \P_m$ is a sequence of ever coarser partitions of $V(G)$ with~$\P_1$ the partition into singletons and $\P_m = \{V(G)\}$ the partition with one part,
      \item $G_i$ is a $\P_i$-flip of $G$ for all $i\in [m]$, and 
      \item for every $i\in [m-1]$ and $v\in V(G)$, we have 
      $$|\Ball^r_{G_{i+1}}(v)/\P_i|\le k.$$
    \end{itemize} 
    A graph class $\C$ has \emph{bounded transient merge-width} if for all $r\in\N$, $\tmw_r(\CC)<\infty$.
\end{definition}
Similarly to the case of merge-width,
equivalently one may require in the definition of transient merge-width that the sequence
$\cal P_1\prec \dots\prec \cal P_n$ is a maximal chain of partitions of $V(G)$.

\begin{example}
To construct a transient merge sequence for the universal cograph $C_m$, take the same sequence of partitions $(\P_i)_{i\in[m+1]}$ as in \cref{ex:run-ex}, and let the $\P_i$-flip $G_i$ be the result of flipping every $P_w\in \P_i$ with itself whenever $w\in V(T_m)$ has odd depth.
	Then, for $i> 1$, the connected components of $G_i$ are precisely the parts of~$\P_{i-1}$, hence $\tmw_r(C_m)=1$ for every $r\in\N$. \qed
\end{example}

Note that every merge sequence gives rise to a transient merge sequence of the same radius-$r$ width.
Indeed if $\P_t$ is homogeneous modulo $R_t$, then there is a $\P_t$-flip $G_t$ of $G$ such that $E(G_t)\subseteq R_t$, and hence
  $\Ball^r_{G_t}(v)\subseteq \Ball^r_{R_t}(v).$  This proves:
\begin{lemma}\label{lem:tmw-mw}
For every graph $G$ and $r\in\N$,
$$\tmw_r(G)\le \mw_r(G).$$
\end{lemma}

Comparing to the definition of merge-width, in transient merge-width there is no consistency required between the graphs considered in consecutive step. Therefore,
pairs of vertices which are close at step $t$ might be far at step $t+1$;
whereas in merge-width, the monotonicity condition $R_1\subseteq\dots\subseteq R_m$
implies that vertices that are close at step $t$ remain so in the subsequent steps.

\subsection{Partition merge-width}
Next, we define the weakest notion, where only  a sequence of partitions is provided.

\begin{definition}
  Let $G$ be a graph and 
    let $\P_1 \preceq\dots \preceq\P_m$ be a sequence of partitions of~$V(G)$, where
  $\P_1$ is the partition into singletons and $\P_m$ has one part.
  The \emph{radius-$r$ width} of the sequence $\cal P_1,\ldots,\cal P_m$
  is defined using flip metrics as:
  $$\max_{1\le i<m}\max_{v\in V(G)} |\Ball_{\P_{i+1}}^{r}(v)/\P_{i}|.$$
    The \emph{radius-$r$ partition merge-width} of $G$, denoted $\pmw_r(G)$, is the minimum radius-$r$ width of a sequence 
    $\cal P_1 \preceq\dots \preceq\P_m$ of partitions of~$V(G)$ as above.
    Finally, a graph class $\C$ has \emph{bounded partition merge-width} if $\pmw_r(\CC)<\infty$ for every $r\in\N$.
\end{definition}
Clearly, in the definition of radius-$r$ partition merge-width one may equivalently consider only sequences $\cal P_1\prec \dots\prec \cal P_m$ where $\P_1$ is the partition into singletons and $\P_m$ has only one part.
\medskip

\begin{example}
	For the universal cograph $C_m$, the analysis for transient merge-width results in a sequence of partitions $(\P_i)_{i\in[m+1]}$ for $C_m$ and $\P_i$-flips $G_i$ which witness that for every $i\leq m$, any two different parts of $\P_i$ are at distance $\infty$ in the flip metric associated with $\P_{i+1}$. Therefore, $\pmw_r(C_m)=1$ for every $r$. \qed
\end{example}

%In fact, for any graph $G$ and its merge sequence $(\cal P_t,R_t)_{t\in [m]}$
%of radius-$r$ width $k$,
%the underlying partition sequence 
%$\cal P_1\preceq\ldots\preceq\cal P_m$ 
%has radius-$r$ width  at most $k$, by \Cref{lem:resolved-flip-metric}. In particular,
%\begin{equation}\label{eq:pmw-mw}
%\pmw_r(G)\le \mw_r(G).  
%\end{equation}
%However, compared to a merge sequence,
%a sequence $\cal P_1\preceq\ldots\preceq\cal P_m$ carries less information, as it  lacks  
%the sets $R_1,\ldots,R_m$ of resolved pairs.

\begin{lemma}\label{lem:pmw-tmw}
For every graph $G$ and $r\in\N$,
$$\pmw_r(G)\le \tmw_r(G).$$
\end{lemma}
\begin{proof}
Assume that 
$(\P_1,G_1),\ldots,(\P_m,G_m)$ is a transient merge sequence of radius-$r$ width $k$, with $\P_1\preccurlyeq \dots\preccurlyeq \P_m$.
Then the sequence $\cal P_1\preccurlyeq \dots\preccurlyeq \cal P_m$ has radius-$r$ width at most $k$,
as for every $v\in V(G)$ and $t\in [m-1]$ we have
\[\Ball^r_{\cal P_{t+1}}(v)\subseteq \Ball^r_{G_{t+1}}(v).\qedhere\]
\end{proof}

 In \Cref{thm:main1} we show that  $\mw_r(G)$ is bounded in terms of $\pmw_{6r+1}(G)$. Hence, bounded merge-width, bounded transient merge-width and bounded partition merge-width are all equivalent properties of graph classes.

\subsection{Definable merge-width}
Finally, we consider partition sequences 
$\P_1\prec \dots \prec \P_n$
in which the partitions~$\P_i$ are \emph{definable} in the following sense: $\P_i$ is the partition into atomic types over some set $S_i$ which forms a prefix of a total order on $V(G)$. This is made precise below.

\begin{definition}
  Let $G$ be a graph and $<$ a total order on $V(G)$, and enumerate~$V(G)$ as $v_1<\dots < v_n$.
  Consider the sequence of partitions $\P_1 \preceq \dots \preceq \P_n$ of $V(G)$, where $\P_i=\Nn_{S_i}$ is the partition into atomic types over the set $S_i\eqdef\{v_1,\dots,v_{n-i}\}$.
  The \emph{radius-$r$ definable merge-width} of $(G,<)$, denoted $\dmw_r(G,<)$, is
  defined as:
  $$\max_{1\le i \le n}\max_{v\in V(G)} |\Ball_{\P_{i}}^{r}(v)/\P_{i}|.$$
  Finally, the \emph{radius-$r$ definable merge-width} of $G$ is minimum of ${\dmw_r(G,<)}$ over all total orders $<$ on $V(G)$.
\end{definition}

\begin{example} \label{ex:defmw}
	For the universal cograph $C_m$, consider an order~$<$ on~$V(C_m)$ so that for each~$i\in [0,m]$, the first~$2^i$ vertices in the order form a transversal of the partition~$\P_{m+1-i}$ (defined in \cref{ex:run-ex}). 
	One can verify that this order witnesses that $\dmw_r(C_m)$ is at most 2 for every $r$.
	 \qed
%	
%	Let $S$ be a non-empty prefix in~$<$ and let~$i$ be the smallest integer such that $|S|< 2^i$. 
%	If $i\leq2$, that is $|S|\leq 3$, then $\Nn_S$ partitions $V(C_m)$ into at most 8 non-singleton parts. 
%	In the $\Nn_S$ metric the singleton classes are at distance~$\infty$ from any other part of~$\Nn_S$, therefore $\max_{v\in V(G)} |\Ball_{\Nn_S}^{r}(v)/\Nn_S|\leq 8$.
%	
%	When $i>2$, we will argue that $\Nn_S$ is a refinement of~$\P_{m+2-i}$.
%	By choice of the order, each part of $\P_{m+1-i}$ contains at most one element of~$S$. 
%	Because $|S|\geq 2^{i-1}$ and as each part~$P\in \P_{m+2-i}$ contains exactly two elements of $\P_{m+1-i}$, $P$ contains one or two elements of~$S$. 
%	As~$i>2$, there are parts $A,B \in \P_{m+2-i}$, different from $P$, such that $AP\subseteq E(C_m)$ and $BP\subseteq E(\overline{C_m})$. 
%	Thus the elements of $A\cap S$, $B\cap S$ show that $P$ must be a union of classes of~$\Nn_S$.
%	
%	For $P\in \P_{m+1-i}$, note that for any $v\in S\setminus P$, $\{v\}$ and $P$ are homogeneous, hence $\Nn_S$ splits $P$ according to the atomic types over $P\cap S$, so into at most 4 non-singleton parts as $|P\cap S|\leq 2$.
%	In the $\Nn_S$ metric, these (at most) 4 parts can be of finite distance only to themselves, because any two different parts of $\P_{m+1-i}$ are homogeneous. This proves that $\max_{v\in V(G)} |\Ball_{\Nn_S}^{r}(v)/\Nn_S|\leq 4$ for $|S|\geq 4$.

\end{example}

%\begin{figure}\label{fig:definable_cograph}
%	\includegraphics[width=0.45\textwidth]{figure/definable_cograph.pdf}
%	\caption{Illustration of the partition $\Nn_S$ of $C_m$ where $S$ is a transversal of $\P_{m-2}$ defined in \cref{ex:defmw}.}
%\end{figure}

Observe that in the above definition, $\P_i$ is obtained from $\P_{i + 1}$ by splitting each part $X \in \P_{i + 1}$ into at most $3$ parts: $X \cap N(v_{n-i})$, $X \cap \overline{N}(v_{n-i})$ and $X \cap \{v_{n - i}\}$.
This implies the following.

\begin{lemma}
  \label{lem:dmw-pmw}
  For every graph $G$ and $r\in\N$,
  $$\pmw_r(G)\le 2 \cdot \dmw_r(G) + 1.$$
\end{lemma}
\begin{proof}
  Let $k = \dmw_r(G)$ and $v_1<\dots < v_n$ be the order on $V(G)$ such that $$\dmw_r(G,<) = k.$$
  Consider the corresponding sequence of partitions $\P_1 \preceq \dots \preceq \P_n$ of $V(G)$, where $\P_i=\Nn_{S_i}$ is the partition into atomic types over the set $S_i=\{v_1,\dots,v_{n-i}\}$.
  % We show that the width of this sequence is at most $2k+1$.
  % First, observe that $\P_i$ is obtained from $\P_{i + 1}$ by splitting each $X \in \P_{i + 1}$ into at most $3$ nonempty parts: $X \cap N(v_{n-i})$, $X \cap N(v_{n-i})$ and $X \cap \{v_{n - i}\}$.
  Let $i < n$ and $v \in V(G)$.
  % For each $X \in \P_{i + 1}$ intersected by $\Ball_{\P_{i + 1}}^r(v)$ there are at most $2$ parts $Y \in \P_i$ other than $\{v_{n-i}\}$ such that $Y \subseteq X$ and $Y \cap \Ball_{\P_{i + 1}}^r(v) \neq \emptyset$.
  For each part $Y \in \P_i$, other than $\{v_{n-i}\}$ and which intersects $\Ball_{\P_{i + 1}}^r(v)$,
  there exists $X \in \P_{i + 1}$ such that $Y = X \cap N(v_{n-i})$ or $Y = X \cap \overline{N}(v_{n-i})$, and $X \cap \Ball_{\P_{i + 1}}^r(v) \neq \emptyset$.
  So the number of such parts $Y$ is bounded by $2k$.
  Additionally, the part $\{v_{n-i}\}$ itself may intersect $\Ball_{\P_{i + 1}}^r(v)$.
  Therefore, $|\Ball_{\P_{i + 1}}^{r}(v)/\P_{i}| \leq 2k + 1$.
\end{proof}
In \Cref{thm:main2} we give an upper bound on $\dmw_r(G)$ expressed as a function of $\mw_{17r+3}(G)$. Together with \Cref{thm:main1}, this will prove \Cref{thm:intro},
that bounded merge-width, definable merge-width, partition merge-width, and transient merge-width are all equivalent properties of graph classes.
\Cref{thm:cw} and \Cref{thm:almost-bounded} will follow similarly, see \Cref{sec:proofs}.

\section{Other Variants}\label{sec:other}
This section discusses three additional variants of merge-width.
The first two --- universal and short merge-width --- are conjectured to be equivalent to the usual notion of merge-width.
This is known in the special case of bounded expansion and bounded twin-width classes.
The third one --- positive merge-width --- is a natural variant, which already captures bounded expansion and bounded twin-width classes.
However, we show that it is not stable under complementation.

\subsection{Universal merge-width}
We start with the concept of universal merge-width, where we require the existence of a single merge sequence witnessing the boundedness of the width for all the radii simultaneously.

\begin{definition}
  A graph class $\CC$ has \emph{universally bounded merge-width}
  if there is a function $f\from\N\to\N$ such that any $G\in\C$ has a merge sequence with radius-$r$ width at most $f(r)$, for every $r\in\N$.
\end{definition}
Clearly, every class of universally bounded merge-width has bounded merge-width. We conjecture the following.

\begin{conjecture}\label{conj:univ}
Every class of bounded merge-width has bounded universal merge-width.
\end{conjecture}
\Cref{conj:univ} is corroborated by the following facts.
First, it follows from the result of van den Heuvel and Kierstead \cite{universal-wcol} (see also \cite[Thm.~3.23]{SIEBERTZ2026100855})
and the construction in \cite[Lem.~7.5]{merge-width} that every class of bounded expansion has universally bounded merge-width. Second, it follows from the construction in \cite[Lem.~7.2]{merge-width} that every class of bounded twin-width has universally bounded merge-width. Finally, it follows from the proof of \cite[Thm.~1.12]{merge-width}
that if $\CC$ has universally bounded merge-width and $\D$ is a first-order transduction of $\CC$, then $\D$ has universally bounded merge-width. In particular, classes of \emph{structurally bounded expansion} \cite{lsd} (first-order transductions of classes of bounded expansion) have universally bounded merge-width. 

In \Cref{thm:quasi-iso-be} we show that every class of universally bounded merge-width 
is quasi-isometric to a class of bounded expansion.

\subsection{Short merge-width}
We now turn our attention to the concept of short merge-width, where we require the merge sequence to be of logarithmic length. However, for technical reasons we also assume a bound on the ``branching'' of a merge sequence, encapsulated in the following parameter: the {\em{valency}} of a merge sequence $(\P_1,R_1),\ldots,(\P_m,R_m)$ is the least $\Delta\in \N$ such that for every $i\in [m-1]$ and every part $P\in \P_{i+1}$, $P$ is the union of at most $\Delta$ parts of $\P_i$.

\begin{definition}Fix $r\in\N$.
  The \emph{radius-$r$ short merge-width} of a graph $G$, denoted $\shortmw_r(G)$, is the least $k\in\N$ such that $G$ has a merge sequence $(\P_1,R_1),\ldots,(\P_m,R_m)$ of radius-$r$ width at most $k$, valency at most $k$, and length \mbox{$m\le k\cdot \log_2(|V(G)|)+k$}.
  We say that a graph class $\CC$ has \emph{bounded short merge-width} if  $\shortmw_r(\CC)<\infty$ for every $r\in\N$ 
\end{definition}

This notion is inspired by \emph{short contraction sequence} considered in the context of twin-width \cite{tww2-journal}. It was shown that classes of bounded twin-width admit such short contraction sequences of bounded twin-width,
which in particular implies that classes of bounded twin-width have bounded short merge-width. 
Also, the construction of logarithmic-length ``block'' orderings for generalised colouring numbers of \cite[Lem.~13]{paraFO} can be combined with the construction of merge sequences from orderings with bounded weak colouring numbers of~\cite[Lemma 7.5]{merge-width} to prove that every class of bounded expansion has bounded short merge-width. %\michal{This statement is a very long shot. Could we make it more detailed?}
Finally, from the results present in the literature it follows that bounded short merge-width is preserved under first-order transductions:
%and from the proof of \cite[Thm.~1.12]{merge-width} that bounded short merge-width is 
%preserved under first-order transductions
%which decrease the size of the graph by at most a polynomial factor.\michal{Isn't this something that we express in a different flavor just below in Lem 3.10?}
\begin{lemma}\label{lem:short-transduction}
Let $\CC$ be a hereditary class of graphs of bounded short merge-width and $\D$ 
be a first-order transduction of $\CC$. Then $\D$ has bounded short merge-width.
\end{lemma}
\begin{proof}[Proof sketch.]
  As $\CC$ has bounded short merge-width, it has bounded merge-width, and thus is monadically dependent by \cite{merge-width}.  
  By \cite[Thm.~40]{quasi-bushes},  if $\D$ is a transduction of a monadically dependent, hereditary class $\CC$, then there is a constant $c$  such that every $H\in \D$ is obtained from some $G\in \CC$ with $|V(G)|\le |V(H)|^c$.
  By the proof of \cite[Thm.~1.12]{merge-width}, it follows that a merge sequence for $G$ of radius-$(r\cdot d)$ width $k$ and length $m$ can be transformed into a merge sequence of $H$ of radius-$r$ width $f(k)$ and length $m$, for some function $f$ and constant $d$ depending on the transduction. Moreover, if the merge sequence for $G$ has valency $\Delta$, then the 
  resulting sequence for $H$ has valency at most $\Delta\cdot d$.
  Then $$m\le k\cdot \log(|V(G)|)+k\le k\cdot \log(|V(H)|^c)+k\le kc\cdot \log(|V(H)|)+kc.$$
  This proves that $\D$ has bounded short merge-width.
\end{proof}
In particular, from all the facts stated above it follows that classes of structurally bounded expansion (that is, first-order transductions of bounded expansion classes) have bounded short merge-width.

All of this motivates us to pose the following.
\begin{conjecture}\label{conj:short}
Every class of bounded merge-width has bounded short merge-width.
\end{conjecture}

In \Cref{thm:adj-labelling} we show that every class of bounded radius-$1$ short merge-width 
admits adjacency labelling schemes with labels of length $\O(\log n)$.

\subsection{Positive merge-width}
Finally, we discuss the notion of positive merge-width, where only edges can be resolved.

\begin{definition}Fix $r\in\N$.
  The \emph{radius-$r$ positive merge-width} of a graph $G$, denoted by $\mw_r^+(G)$,
  is the least number $k\in \N$ such that $G$ has a merge sequence $(\P_1,R_1),\ldots,(\P_m,R_m)$ of radius-$r$ width $k$ such that $R_i\subseteq E(G)$ for all $1 \le i \le n$.
  A graph class $\CC$ has \emph{bounded positive merge-width} if $\mw_r^+(\CC)<\infty$ for every $r\in \N$.
\end{definition}
This notion can be defined equivalently in terms of \emph{construction sequences} as defined in \cite{merge-width}, by considering construction sequences which 
only allow \emph{positive resolves}, such that all edges (rather than all vertex pairs) are resolved in the end.

Classes of bounded positive merge-width are quite general, by the following observations.
It follows from the construction in \cite[Lem.~7.5]{merge-width} that every class of bounded expansion has bounded positive merge-width. 
Furthermore, it is straightforward to modify the construction in \cite[Lem.~7.2]{merge-width} to show that every class of bounded twin-width has bounded positive merge-width.
However, this property is not closed under first-order transductions --- not even under edge complementation --- and not every class of bounded merge-width has bounded positive merge-width. 
 Namely, using the fact that the class of cubic graphs has unbounded twin-width \cite{tww2-journal}, we prove the following. Recall that a {\em{cubic graph}} is one where every vertex has degree~$3$.

\begin{restatable}[]{lemma}{cocubic}\label{lem:co-cubic}
  Let $\CC$ be the class of edge-complements of cubic graphs. Then ${\mw_1^+(\CC)=\infty}$.
\end{restatable}
	\begin{proof}
	We bound the twin-width of each $G\in\CC$ as follows:
	\begin{align}\label{eq:mwpos}
		\tww(G)\le \mw_1^+(G)+3.
	\end{align}
	It is known that the class of cubic graphs has unbounded twin-width \cite{tww2-journal}.
	Since twin-width is preserved under edge complements,
	also $\C$ has unbounded twin-width.
	Together with~\eqref{eq:mwpos}, this proves that $\mw_1^+(\C)=\infty$.
	
	To prove \eqref{eq:mwpos},
	suppose that $G\in\CC$ and $(\P_1,R_1),\ldots,(\P_m,R_m)$ is a merge sequence for $G$ with $R_m\subseteq E(G)$
	and of radius-$1$ width $k$.  
	We may assume that $\P_1\prec \dots \prec \P_m$ is a maximal chain of partitions.
	Since for each $t\in \set{2,\ldots,m}$ we have that $(\P_{t-1},R_{t})$ has radius-$1$ width $k$ and 
	$\P_{t}$ is coarser than $\P_{t-1}$, it follows that 
	$$|\Ball^1_{R_t}(v)/\P_{t}|\le k\quad \text{for all $v\in V(G)$ and $t\in [m]$}.$$

	We argue that $\P_1,\P_2,\ldots,\P_m$ is a contraction sequence 
	witnessing that $G$ has twin-width at most $k+3$.
	
	The key observation is that for every $t\in [m]$ and parts $A,B\in \P_t$,
	if there is some non-edge between $A$ and $B$ in $G$,
	then all edges between $A$ and $B$ must be contained in $R_t$.
	Indeed, as $\P_t$ is homogenous modulo $R_t$, 
	either all edges between $A$ and $B$ are contained in~$R_t$, or all non-edges between $A$ and $B$ are contained in $R_t$. However, as there is a non-edge between $A$ and $B$, the latter possibility is excluded in positive merge sequences by the requirement $R_t \subseteq E(G)$.
	
	Consider $t\in[m]$ and a part $A$ of $\P_t$.
	We argue that $A$ is non-homogeneous towards at most $k+3$ parts $B\in\P_t$ with $A\neq B$.
	Fix $a\in A$.
	If $A$ is non-homogenous towards a part $B\in\P_t$ then in particular there is a non-edge between $A$ and $B$, and by the above, all edges between $A$ and $B$ in $G$ are contained in $R_t$.
	If $a$ has a neighbour $b$ in $B$ then $ab\in R_t$,
	and therefore, the part $B$ accounts to the value $|\Ball^1_{R_t}(a)/\P_{t}|$.
	Hence, there are at most $k$ parts $B$ such that 
	$A$ and $B$ are non-homogeneous and $a$ has a neighbour in~$B$.
	Additionally, there are at most three parts $B\in\P_t$ with $A\neq B$ such that $a$ does not have a neighbour in $B$, as $a$ has at most three non-neighbours altogether (other than $a\in A$).
	In total, there are at most $k+3$ parts $B\in \P_t$ with $A\neq B$ that are non-homogeneous towards $A$.
	Thus, $\P_1,\ldots,\P_m$ is a contraction sequence witnessing that $G$ has twin-width at most $k+3$.
\end{proof}

\section{From Partition Sequences to Merge Sequences}\label{sec:local-metric-conversion}
  This section proves the equivalence of merge-width and partition merge-width.
  The merge sequences constructed in the proof have the following additional useful property.
  Recall that~$\overline{G}$ is the edge-complement of~$G$.
  A merge sequence $(\P_1,R_1),\dots,(\P_m,R_m)$ is called \emph{$d$-firm} if
  for any resolved pair $uv \in R_m$,
  \[ \dist_G(u,v) \le d \quad \text{and} \quad \dist_{\overline{G}}(u,v) \le d. \]
  Note that the same holds for $uv \in R_i$, $i\in[m]$, since $R_i \subseteq R_m$.

\begin{restatable}{theorem}{thmone}\label{thm:main1}
  Let~$G$ be a graph with $\pmw_{6r+1}(G) = k$ and neighbourhood complexity~$\pi_G$. Then
  \[ \mw_r(G)\le k\cdot \pi_G(k). \]
  Moreover, $G$ has a 6-firm merge sequence of radius-$r$ width $k \cdot \pi_G(k)$,
  which can be computed in polynomial time given~$G$ and a sequence witnessing $\pmw_{6r+1}(G)\le k$.
\end{restatable}
Together with \cref{lem:pmw-tmw,lem:tmw-mw}, this proves the equivalence of
bounded merge-width, partition merge-width, and transient merge-width, i.e.\ the first three conditions of \cref{thm:intro}.
\Cref{sec:quasi-iso} presents an application of \Cref{thm:main1} to quasi-isometries and neighbourhood covers, using 6-firm merge sequences.

To prove \cref{thm:main1}, we will use the lemma below, which follows from the proof of \cite[Lem.~8]{flip-separability}.
It uses the following variant of a partition refinement:

Let $G$ be a graph.
Fix a partition~$\P$ of $V(G)$ and a set $S\subseteq V(G)$. 
By the \emph{$S$-refinement} of~$\P$, denoted by ${\P\wedge S}$, we mean the partition
obtained by splitting each $P\in \P$ according to the $S\setminus P$-neighbourhoods. Formally, $u,v\in V(G)$ belong to the same part of~$\P\wedge S$ if and only if they belong to the same part $P\in \P$ and $N(u)\cap (S\setminus P)=N(v)\cap (S\setminus P)$.

\begin{restatable}[{Variant of \cite[Lemma~8]{flip-separability}}]{lemma}{informallemma}
  \label{lem:informal}
  Let~$G$ be a graph, $\P$ a partition of~$V(G)$, and~$S$ a transversal of~$\P$.
  Consider the $S$-refinement $\P' \eqdef \P \wedge S$.
  Then there exists a $\P'$-flip~$G'$ of $G$ such that
\begin{equation*}
  \dist_{\cal P}(u,v)\le 6\quad 
  \text{for all~$uv \in E(G')$}.
\end{equation*}
  Moreover, $G'$ can be computed in polynomial time, given $\cal P$ and $S$.
\end{restatable}
In~\cite{flip-separability}, this lemma was applied to partitions with a bounded number of parts.
When $|\P| = k$, we obtain $|\P'| \le k \cdot 2^k$, and $|\P'| = \O(k^{d+1})$ if~$G$ has VC-dimension~$d$, which yields the version of the statement given in \cite{flip-separability}.
The fact that one can choose $\P'$ to be of the form $\P\wedge S$ for any transversal $S$ of $\P$
is not stated explicitly in \cite{flip-separability}, but follows from a close inspection of the proof,
which we detail in \cref{apx:informal-proof}.

We will apply \cref{lem:informal} to partitions of unbounded size.
The next lemma shows that refining a partition $\P$ by $S$ cannot increase too much 
the width of $\Q$-balls in the refined $\P \wedge S$ metric, where~$\Q$ is any partition refined by~$\P$.
\begin{restatable}{lemma}{mwrefinement}
  \label{lem:mw-neigh-refinement}
  Let~$G $ be a graph, $\P\preceq\Q$ partitions of~$V(G)$, and~$S$ a subset of vertices of~$G$. Let $v\in V(G)$ be such that:
  \begin{enumerate}
    \item $|\Ball_\Q^r(v)/\P| \le k$, and
    \item $|S \cap \Ball^{r+1}_\Q(v)| \le \ell$.
  \end{enumerate}
  Then, for  $\P' \eqdef \P \wedge S$, we have:
  $$|\Ball_{\Q}^r(v)/\P'|\le k \cdot \pi_G(\ell).$$
\end{restatable}
\begin{proof}
	Assume for a contradiction that the ball $B \eqdef \Ball^r_\Q(v)$ intersects more than $k \cdot \pi_G(\ell)$ parts of~$\P'$.
	Since~$B$ intersects only~$k$ parts of~$\P$, there are distinct $Q_1,\dots,Q_t \in \P'$ all contained in the same $P \in \P$, all intersecting~$B$, and with $t > \pi_G(\ell)$. Pick $v_i \in Q_i \cap B$ for $i=1,\ldots,t$.
	By construction of~$\P'$, the vertices $v_1,\dots,v_t$ have pairwise different neighbourhoods in~$S\setminus P$.
	Let $S' \subseteq S\setminus P$ be the set of those vertices $w\in S\setminus P$ such that $w$ is non-homogeneous to $\{v_1,\dots,v_t\}$.
	It follows that the vertices $v_1,\dots,v_t$ have pairwise different neighbourhoods in~$S'$.
	
	As each $s \in S'$ has both a neighbour and a non-neighbour within $\{v_1,\dots,v_t\}$ in $G$, and $v_1,\dots,v_t$ all belong to the same part of $\P \preceq \Q$,
	the same holds for any $\Q$-flip of $G$. Therefore, $S' \subseteq \Ball^{r+1}_\Q(v)$ so $|S'| \leq \ell$.
	Because vertices $v_1,\dots,v_t$ have pairwise distinct neighbourhoods in~$S'$, it follows that $t \leq \pi_G(|S'|) \leq \pi_G(\ell)$.
	This contradicts the assumption.
\end{proof}

As an immediate consequence, we obtain:
\begin{corollary}\label{cor:mw-transversal-refinement}
	Let~$G $ be a graph, $\P\preceq\Q$ partitions of~$V(G)$, and~$S$ a transversal of $\P$.
  Suppose~$|\Ball_\Q^{r+1}(v)/\P| \le k$ for every~$v\in V(G)$.
  Then for every $v\in V(G)$, we have 
  $$|\Ball_{\Q}^r(v)/(\P \wedge S)|\le k \cdot \pi_G(k).$$
\end{corollary}

\begin{proof}[Proof of \Cref{thm:main1}]
Fix a sequence  $\P_1 \preceq\dots \preceq\P_m$ of partitions of $V(G)$ such that $\cal P_1$ is the partition into singletons and $\cal P_m$ has one part.
We may pick 
  transversals $S_1,\ldots, S_m$ of $\P_1,\dots,\P_m$ respectively,
  so that ${S_1\supseteq \ldots \supseteq S_m}$.
  Let $\P'_i \eqdef \P_i \wedge S_i$.
  Note that $S_i \supseteq S_{i+1}$ implies that~$\P'_i$ refines~$\P'_{i+1}$.
  Moreover,~$\P'_1\preceq\P_1$ is the partition into singletons and $\P'_m=\{V(G)\}$.
  By \cref{lem:informal}, 
  for each $i\in [m]$
  there is a $\P'_i$-flip~$G_i$ of~$G$ such that
  \begin{equation}
    \label{eq:informal-lemma}
    \dist_{\cal P_i}(u,v)\le 6\qquad\text{for all  $uv\in E(G_i)$.}
  \end{equation}
  Define, for each $i\in [m]$, $$R_i \eqdef \bigcup_{j \le i} E(G_j).$$
  Since~$R_i$ is a superset of the edge set of a $\P'_i$-flip, $\P'_i$ is homogenous modulo $R_i$.
  Also, we have $R_i \subseteq R_{i+1}$ by construction. Furthermore, the merge sequence $(\cal P_1',R_1),\ldots,(\cal P_m',R_m)$ can be computed in polynomial time, given $\cal P_1,\ldots,\cal P_m$.

  Let us first check that this new merge sequence is 6-firm.
  We in fact obtain a stronger property:
  \begin{equation}\label{eq:extra-firm}
    \text{for any $i \in [m]$ and $uv \in R_i$,} \quad \dist_{\P_i}(u,v) \le~6.
  \end{equation}
  Indeed, if $uv \in R_i$, then there is some $j \le i$ such that $uv \in E(G_j)$.
  Then we have $\dist_{\P_i}(u,v) \le \dist_{\P_j}(u,v)\le 6$,
  where the first inequality comes from \cref{lem:flip-metric-refine} since $\P_j$ refines~$\P_i$,
  and the second is given by \cref{eq:informal-lemma}.
  When $i=m$, this means that for any $uv \in R_m$, $u$ and~$v$ are at distance at most~6 in any $\P_m$-flip, i.e.\ in~$G$ and~$\overline{G}$, proving 6-firmness.

  Furthermore, \cref{eq:extra-firm} can be restated as: for every $r\in \N$, $i\in [m]$, and $v\in V(G)$, we have
  \begin{equation}
    \Ball_{R_i}^r(v) \subseteq \Ball_{\P_i}^{6r}(v).
    \label{eq:balls-apx}
  \end{equation}
  % \karolina[inline]{I would write $\Ball_{R_i}^r(v) \subseteq \Ball_{G_i}^r(v)$}

  Now let 
   $r,k\in\N$ and suppose that $\cal P_1,\ldots,\cal P_m$ has radius-$(6r+1)$ width $k$, i.e. $$|\Ball_{\P_{i+1}}^{6r + 1}(v)/\P_{i}|\le k$$ for every $v\in V(G)$ and $i\in [m-1]$. 
By \cref{cor:mw-transversal-refinement}, this implies $$|\Ball_{\P_{i+1}}^{6r}(v)/\P_{i}'|\leq k \cdot \pi_G(k).$$
  Then \cref{eq:balls-apx} gives $$|\Ball_{R_{i+1}}^r(v)/\P'_i| \le k \cdot \pi_G(k),$$
  so $(\P'_i,R_{i+1})$ has radius-$r$ width at most $k \cdot \pi_G(k)$. This concludes the proof of \Cref{thm:main1}.
\end{proof}

\section{Constructing Definable Merge Sequences}\label{sec:defmw}

The main result of this section is the following.
\begin{theorem}\label{thm:def-mw}
	Let $G$ be a graph with $\mw_{17r+3}(G) \leq k$, of VC-dimension~$d$, and neighbourhood complexity $\pi_G$. Then there is some
    $M\le \O(dk^3)$ %\karolina{$m$ also denotes the length of merge sequence}
    and a total order on $V(G)$ such that for any prefix $S$ of this order and any $v\in V(G)$, there exists an $S$-flip $G'$ of $G$ such that
	\[|\Ball_{G'}^r(v)/ \Nn_S| \le  M + k\cdot \pi_G(M)\le  (dk)^{\O(d)}.\]
\end{theorem}
This immediately yields the following:
\begin{corollary}\label{thm:main2}  For every $r\in\N$ and graph $G$ of VC-dimension $d$,
   $$\dmw_r(G)\le d^{\Oh(d)}\cdot \mw_{17r+3}(G)^{\O(d)}.$$
\end{corollary}

Note that together with \cref{thm:main1} and \cref{lem:dmw-pmw,lem:pmw-tmw,lem:tmw-mw}, \cref{thm:main2} completes the proof of \cref{thm:intro}.

The construction in our proof of \Cref{thm:def-mw} is based on \cite[Thm.~3.3]{local-cliquewidth},
which considered a single partition $\P$, and showed how to choose a ``sample set'' $S$ so that the distances in the $\Nn_S$-metric are not much smaller than in the $\P$-metric. 
We lift this approach to a sequence of partitions obtained from a merge sequence. We start with introducing the necessary definitions and results based on  \cite{local-cliquewidth}.

\subsection{Sample sets and definable flips}
Sample sets consist of two parts: \emph{witnesses} and \emph{duals}.

\paragraph{Small and large parts, witnesses}
% By a \emph{metric} on a set $V$ we mean a symmetric function $\dist\colon V^2\to \mathbb R_{{\ge}0}\cup\set{\infty}$ satisfying the triangle inequality, and such that $\dist(v,v) = 0$ for all $v \in X$. 
% By $\Ball^r(v)$ we denote 
% $\setof{w\in V}{\dist(v,w)\le r}$.
% We will only consider metrics of the form $\dist_R$, where $(V,R)$ is a graph and $R\subset \binom{V}{2}$ and for $u,v\in V$, and $\dist_R(u,v)$ is the length of a shortest path from $u$ to $v$ in the graph $(V,R)$, or $\infty$ if no such path exists.

Let~$V$ be a vertex set and $R\subset \binom{V}{2}$ be a set of vertex pairs (which will be a set of resolved pairs in our setting).
Let $s,t \in \N$ with~$s \le t$. We say that a set $X\subset V$ is:
\begin{enumerate}
  \item \label[case]{it:setS_i} \emph{$t$-small} if there exists $x_1 \in X$ such that $X \subseteq \Ball^t_R(x_1)$.
  \item \label[case]{it:setS_ii} \emph{$(s,t)$-medium} if there are $x_1,x_2 \in X$ with $\dist_R(x_1,x_2) > t$, and $X \subseteq \Ball^s_R(x_1) \cup \Ball^s_R(x_2)$.
  \item \label[case]{it:setS_iii} \emph{$s$-large} if there are $x_1,x_2,x_3 \in X$ that are pairwise at distance more than~$s$ in $(V,R)$.
\end{enumerate}
In all three cases, vertices~$x_i$ are called \emph{witnesses} that~$X$ is small, medium, or large.
These three properties are not mutually exclusive, but one of the three must hold.
Indeed, witnesses that~$X$ is small, medium, or large can be found greedily:
pick any~$x_1$, then any~$x_2$ at distance more than~$t$ from $x_1$, and any~$x_3$ at distance more than~$s$ from both $x_1$ and $x_2$, stopping if the step is impossible.

\paragraph{Duals}
For a graph $G$ and subsets of vertices $X, Y \subseteq V(G)$, a set $S_{XY}$ is \emph{dual} for $(X,Y)$ if:
\begin{itemize}
	\item $S_{XY} \subseteq X$ and $S_{XY}$ dominates $Y$, i.e. $Y\subseteq N(S_{XY})$, or
	\item $S_{XY} \subseteq Y$ and $S_{XY}$ anti-dominates $X$, i.e. $X\subseteq \overline{N}(S_{XY}).$
\end{itemize}
A graph $G$ is said to have \emph{a duality of order} $d\in \mathbb{N}$ if every pair of subsets of $V(G)$ has a dual of size at most $d$.
The following result is an improvement of \cite[Thm.~3.3]{local-cliquewidth}, which itself follows from the $(p,q)$-theorem \cite[Thm.~4]{matousek2004VC}.
\begin{theorem}[{\cite[Thm.~E.2]{flip-width}}]\label{thm:dual_vc}
	Any graph $G$ has a duality of order $\O(\VC(G))$.
\end{theorem}

\paragraph{Sample sets}
Consider a graph $G$, a partition $\P$ of $V(G)$ which is homogeneous modulo a set $R\subset \binom{V}{2}$,
and let $s \le t$.
A set $S\subseteq V(G)$ is an \emph{$(s,t)$-sample set} of $(\P,R)$ if
\begin{enumerate}
  \item for every $X \in \P$, $S$ contains witnesses that~$X$ is either $s$-large or $t$-small with respect to $R$; and
  \item $S$ contains a dual $S_{XY}$ for every $(X,Y)\in \P^2$.
\end{enumerate}
Note that this definition forbids medium parts in~$\P$.
To find a sample set for a given partition~$\P$, one should thus first split each medium part into two small parts, before looking for witnesses and duals (see \Cref{lem:partitions}).

The next result follows from the proof of~\cite[Thm.~3.5]{local-cliquewidth}, where the result was shown for a $(2,2)$-sample set; the lift to $(2,t)$-sample sets follows from a straightforward inspection of the reasoning.
Recall from \cref{sec:partitions,sec:flips} that for~$S\subseteq V(G)$,~$\Nn_S$ is the partition of~$V(G)$ according to the atomic types over~$S$, and that an \emph{$S$-flip} is short for an $\Nn_S$-flip.

\begin{lemma}[\cite{local-cliquewidth}]\label{lem:incremental}
  Consider a graph $G$ and a partition~$\P$ of $V(G)$ which is homogeneous modulo a set $R\subset \binom{V}{2}$. For every $(2,t)$-sample set $S\subseteq V(G)$ for $(\P,R)$ there is an $S$-flip $G'$ of $G$ such that for all $u,v \in V(G)$,
  \[ \dist_{G'}(u,v)  \geq \dist_{R}(u,v)/(2t+1); \] or equivalently, for every $r\in \N$ and $v\in V(G)$,
  \[\Ball_{G'}^r(v) \subseteq \Ball_R^{(2t+1)r}(v).\]
\end{lemma}
Note that \cite[Thm.~3.5]{local-cliquewidth} actually states that for any $P,Q \in \Nn_S$, pairs of vertices $u\in P, v \in Q$ with $\dist_R(u,v) > 2t+1$ are either all edges or all non-edges in $G$. To obtain the desired $S$-flip $G'$, flip each pair $(P,Q)\in \Nn_S^2$ such that all pairs $uv\in PQ$ with $\dist_R(u,v)>2t+1$ are edges in $G$; then the conclusion of \Cref{lem:incremental}~follows.

\subsection{Finding sample sets in merge sequences}
We would like to apply \cref{lem:incremental} to each step of a merge sequence.
This requires constructing a monotone sequence of sample sets for its partitions, while ensuring that balls of fixed radius do not contain too many points from the sample set.

\begin{restatable}{lemma}{samplesets} \label{lem:orderS}
  Let $r\in \N$ and $G$ be a graph with duality of order $d$.
  Let $$(\P_1, R_1), \dots , (\P_m , R_m)$$ be a merge sequence for $G$ of radius-$(r+2)$ width at most $k$
  such that
  $\P_1\prec \dots\prec \P_m$ is a maximal chain of partitions.
  Then there exist sets $V(G) = S_0 \supseteq S_1 \supseteq \dots \supseteq S_{m-1}$ such that for every $i\in [m-1]$,
  \begin{enumerate}
    \item  $|\Ball_{R_{i+1}}^{r}(v) \cap S_{i-1}| \le  \O(dk^3),$ for all $v \in V(G)$, and
    \item there is a partition~$\P'_i$ obtained by possibly splitting each part of~$\P_i$ in two, such that $S_i$ is a $(2,8)$-sample set for $(\P'_i, R_{i+1})$.
  \end{enumerate}
\end{restatable}

Before giving the proof, let us briefly elaborate on the approach and the difficulties. The construction of the sample sets $S_0 \supseteq S_1 \supseteq \dots \supseteq S_{m-1}$ is quite natural: each sample set $S_i$ consists of a set $W_i$, containing a witness for every part of $\P_i$, and a set $D_i$, containing a dual set for every pair of parts of $\P_i$. The sets $W_i$ and $D_i$ are constructed by scanning the merge sequence in the reverse order: we start with $W_m=D_m=\emptyset$ that work for $\P_m=\{V(G)\}$, and iteratively revert the merges along the maximal chain of partitions $\P_m\succ \ldots \succ \P_1$. When proceeding from $\P_i$ to $\P_{i-1}$, the partition gets finer, which means that we might need to add some new witnesses to $W_i$ in order to get a suitable~$W_{i-1}$, and augment some dual sets in $D_i$ in order to get a suitable~$D_{i-1}$. We do this in any inclusion-wise minimal way. Note that this may result in adding an unbounded number of new vertices to the constructed sample set $S_{i-1}$. However, here is the crucial point: the inclusion-wise minimality of the augmentation can be used to argue that $S_{i-1}$ is still ``sparse'', in the sense that every distance $r$-ball in $R_{i+1}$ contains only a bounded number of elements of $S_{i-1}$; this establishes the first assertion from the lemma statement. For the second assertion, in order to get rid of the medium parts we artificially split them into two small parts along the way, thus obtaining partitions $\P'_i$ from $\P_i$; and we work with partitions $\P'_i$ instead. This adds an extra layer of technicalities to the argument.

\begin{proof}
	We first give an overview of the proof.
	We will start with constructing sets $W_1 \supseteq \dots \supseteq W_{m-1}$,
	where $W_i$ contains witnesses that each part of $\P_i$ is small, medium, or large.
	This will give us a sequence of partitions $\P_1', \dots ,\P_m'$, where $\P_i'$ is obtained from $\P_i$ by splitting each medium part into two.
	Then, we will construct sets $D_1 \supseteq \dots \supseteq D_{m-1}$ such that $D_i$ contains dual sets for each pair $X', Y' \in \P_i'$.
	Having defined $W_i$ and $D_i$, we define: 
	\begin{equation}\label{def:S_i}
		S_i \eqdef W_i \cup D_i\qquad\text{for $i\in [m-1]$.}
	\end{equation}
	By construction, $S_1\supseteq\dots\supseteq S_{m-1}$ will be a monotone sequence of sample sets for $(\P'_i, R_{i+1})$. Our goal will be to prove that $\Ball^r_{R_{i+1}}(v)$ contains few vertices of~$S_{i-1}$, for all $v\in V(G)$ and $i\in[m-1]$; this will conclude the lemma.
	
	\medskip
	{\itshape
		Throughout the remainder of the proof of \Cref{lem:orderS}, we fix 
		a merge sequence $$(\P_1, R_1), \dots, (\P_m, R_m)$$ of radius $(r+2)$-width at most $k$, such that $\P_1\prec\dots\prec \P_m$ is a maximal chain of partitions.}
	
	\medskip
	
	The sets $W_1,\ldots,W_{m-1}$ are constructed in the following lemma.
	\begin{lemma}\label{lem:witnesses}
		% Let $(\P_1, R_1), \dots, (\P_m, R_m)$ be a merge sequence for $G$
		% of $r$-width at most $k$ such that $\P_1 \prec \dots \prec \P_m$ is a maximal chain of partitions.
		There are subsets $W_1 \supseteq \dots \supseteq W_{m-1}$ of $V(G)$ such that for every $i \in [m-1]$, for every part $X \in \P_i$, the set $X \cap W_i$ contains witnesses that~$X$ is $8$-small, $(2,6)$-medium, or $2$-large with respect to~$R_{i+1}$,
		and furthermore, for $i > 1$ and every $v\in V(G)$, \[|\Ball_{R_{i+1}}^r(v) \cap W_{i - 1}| \le 3(k + 1).\] %$|X \cap W_i| \le 3$.    
	\end{lemma}
	
	\begin{proof}
		We construct the sets~$W_i$ for $i\in [m]$ in reverse order, i.e.\ starting with~$W_{m}$, so that the following hold for all $i\in [m-1]$ and $X\in\P_i$:
		\begin{align} 
			&\text{ $\dist_{R_{i+1}}(x,y)>2$ for all $x,y\in X \cap W_i$, and } \label{prop:Mi}\\
			&\text{$|X \cap W_i|\le 3$.} \label{prop:Mi2}
		\end{align}
		Set $W_m=\emptyset$.
		Let $i\in[m-1]$, and 
		assume~$W_{i+1}$ with the required properties is already constructed.
		Since~$\P_i$ refines~$\P_{i+1}$, and distances in~$R_{i+1}$ are at least as large as in~$R_{i+2}$,
		% \jakub[]{changed from "distances in~$R_{i+1}$ are at least as large as in~$R_{i+2}$"}
		this implies that for every part $X\in \P_i$, the set $X \cap W_{i+1}$ consists of at most~3 vertices pairwise at distance more than~2 in~$R_{i+1}$.
		Then $X \cap W_i$ is defined by greedily picking vertices from~$X$ pairwise at distance more than~$2$ in~$R_{i+1}$, starting with $X \cap W_{i+1}$, and ending when either 3 such vertices have been found, or no further vertex can be added.
		Thus, conditions \eqref{prop:Mi}, \eqref{prop:Mi2}, and $W_i \supseteq W_{i+1}$ are immediately satisfied.
		
		Now consider any part $X \in \P_i$.
		If $X \cap W_i$ has three vertices, then they are witnesses that~$X$ is 2-large.
		If $X \cap W_i = \{x\}$ is a singleton instead, then $X \subseteq \Ball^2_{R_{i+1}}(x)$ (or another vertex would have been added), so~$x$ witnesses that~$X$ is 2-small.
		Finally, if $X \cap W_i = \{x,y\}$, then we again have that~$X$ is contained in~$\Ball^2_{R_{i+1}}(x) \cup \Ball^2_{R_{i+1}}(y)$.
		If $\dist(x,y) > 6$, then~$x,y$ witness that~$X$ is $(2,6)$-medium,
		while if~$\dist(x,y) \le 6$, then just~$x$ witnesses that~$X$ is 8-small.
		
		Note that, for $i > 1$, each ball $\Ball_{R_{i+1}}^r(v)$ intersects at most~$k + 1$ parts of~$\P_{i - 1}$ (as it intersects at most $k$ parts of $\P_i$ and $\P_{i - 1}$ is obtained from $\P_i$ by splitting some part into two).
		Each of these parts contains at most~3 vertices of~$W_{i-1}$,
		hence 
		\begin{equation*}\label{eq:witnesscount}
			|\Ball_{R_{i+1}}^r(v) \cap W_{i - 1}| \le 3(k + 1).\claimqedhere
		\end{equation*}
	\end{proof}

  \begin{figure}[h]
    \centering
    \includegraphics[scale =2.1]{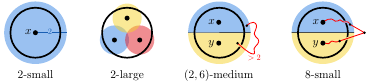}
    \caption{Illustration of $4$ differents cases for witnesses of a part $X\in \P_i$. Colored blobs correspond to radius-$2$ balls of the witnesses in $R_{i+1}$, \textcolor{red}{red} edges and paths are in $R_{i+1}$. %From left to right the part are respectively, $2$-small, $2$-large, $(2,6)$-medium, $8$-small. 
    }
    \label{fig:witnesses}
  \end{figure}

	{\itshape In the remainder of the proof of \Cref{lem:orderS}, we fix a sequence $W_1\subseteq \dots\allowbreak \subseteq W_{m-1}$ as obtained in \Cref{lem:witnesses}.}
	\medskip
	
	We now construct a sequence of partitions $\P_1', \dots, \P_m'$.
	
	Let $i \in [m - 1]$. For $X \in \P_i$
	choose $W_i^X \subseteq W_i \cap X$ to be witnesses for~$X$ being $8$-small,
	$(2,6)$-medium, or $2$-large in with respect to~$R_{i+1}$.
	% (Note that possibly $W_i^X \subsett W_i \cap X$.)
	From now on, we call a part~$X$  \emph{small}, \emph{medium} or \emph{large} (omitting the radius parameters) according to the choice made for~$W_i^X$.
	Recall that the definition of sample sets forbids medium parts, and we thus want to split them into small parts.

	For $i \in [m-1]$, let $\P_i'$ be obtained from $\P_i$ by splitting each medium part $X \in \P_i$ with witnesses $W_i^X = \{x_1,x_2\}$
	into two parts $X^1= \Ball^2_{R_{i+1}}(x_1) \cap X$ and $X^2= \Ball^2_{R_{i+1}}(x_2) \cap X$, called the \emph{halves} of $X$.
	Note that these halves are disjoint since we require $\dist_{R_{i+1}}(x_1,x_2) > 6$.
	\begin{lemma}\label{lem:partitions}
		\crefalias{enumi}{lemma}
		For each $i\in[m-1]$, the following properties hold:
		\begin{enumerate}[label=(\arabic*), ref=(\arabic*)]
			\item\label{prop:sample}  $W_i$ contains witnesses for each part of $\P_i'$ being either $8$-small or $2$-large with respect to $R_{i+1}$.
			\item\label{prop:medium-halves} For every medium part $X \in \P_i$, if two vertices are in different halves of $X$, then they are at distance more than $2$ in~$R_{i+1}$.
			\item\label{prop:medium-halves-future} For every $i \leq j < m$, if $X \in \P_i$ and $Y \in \P_j$ are medium parts such that $X \subseteq Y$, if two vertices are in the same half of $X$, then they are also in the same half of $Y$.
			\item\label{prop:halves-pigeonhole} For every part $X \in \P_i$ and subset $A \subset X$, for every weight function $f \colon A \to \N$ with $\sum_{x \in A} f(x)>2\ell$, there is a subset $B \subseteq A$ such that $\sum_{x \in B} f(x)>\ell$ and $B$ is contained in a single part of~$\P'_j$, for all $j \ge i$.
			% \item For every part $X \in \P_i$ and subset $A \subset X$ with $|A|>2\ell$, there is a subset $B \subseteq A$ with $|B|>\ell$ such that $B$ is contained in a single part of~$\P'_j$, for all $j \ge i$.\label{prop:halves-pigeonhole}
		\end{enumerate}
	\end{lemma}

	\begin{proof} Let us prove the four conditions.
		\begin{enumerate}
			\item For small and large parts of~$\P_i$, this condition is immediate by the choice of~$W_i$.
			Consider a medium part $X \in \P_i$ and $W_i^X = \{x_1,x_2\}$ its witnesses.
			The corresponding parts in~$\P'_i$ are the halves $X^s = \Ball^2_{R_{i+1}}(x^s)$ for $s=1,2$, which are 2-small, as witnessed by~$x_s \in W_i \cap X^s$.
			Therefore, for each $X' \in \P'_i$, the set $W_i \cap X'$ contains witnesses that~$X'$ is either 8-small or 2-large.
			
			\item For a medium part $X \in \P_i$ with witnesses $W_i^X = \{x_1,x_2\}$, recall that being $(6,2)$-medium requires $\dist_{R_{i+1}}(x_1,x_2) > 6$.
			On the other hand, the two halves~$X^1,X^2$ are contained in the 2-balls in~$R_{i+1}$ around~$x_1,x_2$.
			Then triangle inequality gives that $\dist_{R_{i+1}}(u,v) > 2$ for any $u \in X^1$, $v \in X^2$.
			
			\item For $i \le j < m$, consider medium parts $X \in \P_i$ and $Y \in \P_j$ with $X \subseteq Y$.
			Take~$u,v$ in the same half of~$X$, say $u,v \in \Ball^2_{R_{i+1}}(x)$ where $x \in W_i^X$ is one of the witnesses of~$X$ being medium.
			As $R_{i+1} \subseteq R_{j+1}$, we also have that~$u,x$, respectively~$v,x$ are at distance at most~2 in~$R_{j+1}$.
			By condition~\ref{prop:medium-halves}, this implies that~$u,x$, respectively~$v,x$ are in the same half of~$Y$.
			So~$u,v$ are also in the same half of~$Y$, as~desired.
			
			\item Let~$X \in \P_i$ be a part, and for any~$j \ge i$, call~$X_j$ the part of~$\P_j$ containing~$X$.
			Consider the minimal~$j$ such that~$X_j$ is medium. (If there is no such $j$, then the claim is trivial with $B=A$.)
			For the given weight function~$f$, one of the two halves of~$X_j$ has weight more than~$\ell$, say $\sum_{x \in A \cap X_j^1} f(x) > \ell$.
			Then choose~$B$ to be $A \cap X_j^1$.
			For any subsequent $j' \ge j$ where~$X_{j'}$ is medium, condition~\ref{prop:medium-halves-future} implies that~$B$ is contained in a single half of~$X_{j'}$,
			and when~$X_{j'}$ is not medium, we trivially have $B \subset X_{j'} \in \P_{j'}$.
		\end{enumerate}
		This concludes the proof of \cref{lem:partitions}.
	\end{proof}
	Let us point out that while condition~\ref{prop:medium-halves-future} above proves that
	medium parts $X \in \P_i$, $Y \in \P_j$ are split in a consistent way,
	there might still be at some intermediate step $i < \ell < j$ a part $Z \in \P_\ell$ with $X \subseteq Z \subseteq Y$ that is not medium, and thus is not split in~$\P'_\ell$.
	Thus, it is not the case that $\P'_i$ refines $\P'_{i+1}$.
	Condition~\ref{prop:halves-pigeonhole} is a form of the pigeonhole principle that helps us to work around this difficulty.
	
	We now fix the sequence $\P_1',\dots,\P_m'$ as described above,
	and are ready to construct the sets $D_1 \supseteq \dots \supseteq D_{m-1}$.
	\begin{lemma}\label{lem:duals}
		There exist sets $D_1 \supseteq \dots \supseteq D_{m-1}$ such that for every $i \in [m-1]$, $D_i$ contains a dual set for each pair of parts $X', Y' \in \P_i'$,
		and moreover, for $i > 1$ and every $v\in V(G)$, 
		\begin{equation*}\label{eq:boundD}
			|\Ball_{R_{i+1}}^r(v)\cap D_{i - 1}|\le \O(dk^3).  
		\end{equation*}    
	\end{lemma}
	\newcommand{\ins}{\mathsf{ins}}
	\begin{proof}%\sz{this proof still needs polishing}\michal{Went through it. It was a hard read, but I think it is ok.}
		We construct dual sets~$D_i$ for parts in~$\P'_i$, again in reverse order, starting with $D_m=\emptyset$.
		Let $i\in[m-1]$ and assume that~$D_{i+1}$ is already defined.
		We define~$D_i$ by the following process. Start with $D_i\coloneqq D_{i+1}$.
		Next, we consider all pairs $(X',Y')$ of parts of~$\P'_i$ in any order, and for each such pair, we add to $D_i$  any inclusion-wise minimal set~$\Delta$ which together with the current~$D_i$ contains a dual set for~$(X',Y')$.
		Then $|\Delta|\le d$ by the assumption that $G$ has a duality of order~$d$.
		We call~$\Delta$ a \emph{batch} from \emph{generation~$i$}, and all the elements of $\Delta$ are \emph{generated by the pair $(X',Y')$}.
		
		This process constructs~$D_1,\dots,D_{m-1}$ in reverse order,
		by sequentially adding vertices in batches (the~sets $\Delta$ as above).
		This thus defines an \emph{insertion quasi-ordering} $\qlt_\ins$ on $D_{1}$,
		where $x \qlt_\ins y$ means that~$x$ was inserted in some batch strictly before the batch containing~$y$,
		while $x \qle_\ins y$ means that~$x$ was inserted before or in the same batch as~$y$.
		Note that insertion times are reversed compared to time in the merge sequence,
		i.e.\ if~$x,y$ are from generations~$i,j$ respectively with $i < j$, then $x \qgt_\ins y$.

		Let $i \in [m-1]$.
		Consider some vertex~$x$ from generation~$i$, contained in the part $X' \in \P'_i$. This means that $x$ was generated by the pair $(X',Y')$ or  the pair $(Y',X')$, for some $Y' \in \P'_i$.
		This means~$x$ belongs to an inclusion-wise minimal set~$\Delta$
		such that~$\Delta$, together with vertices of~$D_i$ inserted before~$x$,
		contains a subset of~$X'$ that dominates~$Y'$ (if $x$ is generated by $(X',Y')$) or anti-dominates $Y'$ (if~$x$ is generated by $(Y',X')$).
		Thus, the minimality of~$\Delta$ implies that there is some vertex $w\in Y'$ that is adjacent (resp. non-adjacent) to~$x$,
		and not to any other vertex  $y\in D_i\cap X'$ with $y\qle_\ins x$. We call $w$ the \emph{reason} for~$x$, and denote it by $\rho(x)$ (see \cref{fig:duals}).
		Precisely, we have the following:

		\begin{claim}\label{clm:reason-minimal}
			Let $i \in [m-1]$, $x \in D_i$ be a vertex from generation~$i$, and  $X' \in \P'_i$ be the part containing~$x$.
			Then for any $y \qle_\ins x$ with $y \in X'\setminus \set x$, $$x\rho(x)\in E(G)\iff y\rho(x)\notin E(G).$$
		\end{claim}
		
    \begin{figure}[h]
      \centering
      \includegraphics[scale=2.2]{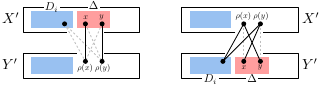}
      \caption{Illustration of the construction of the dual set for $(X',Y')$ with $X',Y'\in \P_i$.}\label{fig:duals}
    \end{figure}

		Our next goal is to prove that for $i > 1$ and $v\in V(G)$,
		\begin{equation}\label{eq:mainbound}
			\left|\Ball_{R_{i+1}}^r(v)\cap D_{i - 1}\right|\le 8d(k+1)^2(k+3) + 2(k+1).  
		\end{equation}
		Fix $i>1$ and a vertex $v \in V(G)$, and denote $B^r \eqdef \Ball_{R_{i+1}}^r(v)$.
		Assume towards a contradiction that
		$$\left|D_{i - 1} \cap B^r\right| > 8d(k+1)^2(k+3) + 2(k+1).$$
		
		By the assumption on the width of the merge sequence, $B^r$ intersects at most~$k$ parts of~$\P_i$, so at most $k+1$ parts of $\P_{i-1}$.
		Thus, there is some part $X \in \P_{i - 1}$ such that $$\left|D_{i - 1} \cap B^r \cap X\right| > 8d(k+1)(k+3)+2.$$
		Now, by \cref{lem:partitions}, condition \ref{prop:halves-pigeonhole} (with $A=D_{i - 1} \cap B^r \cap X$ and $f(v) \eqdef 1$ for each $v\in A$), there is a subset
		\begin{equation}
			\Gamma_0 \subseteq D_{i-1} \cap B^r \cap X,
		\end{equation}
		with 
		\begin{equation}\label{eq:sizeofgamma}
			\left|\Gamma_0\right|>4d(k+1)(k+3)+1,  
		\end{equation}
		such that the following holds:
		\begin{equation}\label{eq:Gamma-same-part}
			\textit{$\Gamma_0$ is contained in a single part of~$\P'_j$, for all~$j \ge i - 1$.}
		\end{equation}
		
		Now consider any pair of distinct vertices $x, y \in \Gamma_0$ such that $y \qle_\ins x$.
		Let~$j$ be the generation of $x$, which is at least~$i-1$ since $x \in D_{i-1}$.
		By \cref{eq:Gamma-same-part}, $x,y$ are in the same part of~$\P'_j$.
		% For any $y \qle_\ins x$ in~$\Gamma_0$, if~$x$ is from generation~$j$, then $j \ge i - 1$, and we have by \eqref{eq:Gamma-same-part} that~$x,y$ are in the same part of~$\P'_j$.
		Thus, \cref{clm:reason-minimal} implies that $\rho(x)$ has both an edge and a non-edge towards~$\{x,y\}$ in~$G$.
		Since~$x,y$ are both in~$X \in \P_{i - 1}$, and~$\P_{i-1}$ is homogeneous modulo~$R_{i-1}$,
		this implies that either~$x\rho(x)$ or~$y\rho(x)$ belongs to~$R_{i-1}\subset R_{i+1}$.
		Thus,~$\rho(x)$ is a neighbour of~$\Gamma_0$ in~$R_{i+1}$, and this holds for all~$x\in \Gamma_0$, except possibly the minimum of~$\Gamma_0$ if it is unique
		(i.e.\ the~unique vertex $x \in \Gamma_0$ such that $x \qlt_\ins y$ for every $y \in \Gamma_0$, if it exists).
		
		Recall from \cref{eq:sizeofgamma} that~$\left|\Gamma_0\right|>4d(k+1)(k+3)+1$.
		Thus, if we remove from~$\Gamma_0$ the smallest element,
		we are left with a subset $\Gamma_1 \subseteq \Gamma_0$ with $|\Gamma_1|>4d(k+1)(k+3)$ and $\rho(\Gamma_1)\subset N_{R_{i+1}}(\Gamma_0)$.
		
		Combined with \cref{eq:Gamma-same-part}, we also obtain that for any $j \ge i - 1$,
		if~$X'_j$ denotes the part of~$\P'_j$ containing~$\Gamma_0$, then:
		\begin{equation}\label{eq:Gamma-reason-close}
			\rho(\Gamma_1)\subseteq N_{R_{i+1}}(X'_j).
		\end{equation}
		
		Since $\Gamma_0 \subseteq B^r$ and $\rho(\Gamma_1)\subset N_{R_{i+1}}(\Gamma_0)$, it follows that~$\rho(\Gamma_1) \subseteq B^{r+1} \eqdef \Ball_{R_{i+1}}^{r+1}(v)$.
		As $|B^{r+1}/\P_{i-1}|\le k + 1$, there is a part $Y \in \P_{i - 1}$ and a subset $\Gamma_2\subset \Gamma_1$
		such that the following holds:
		\begin{equation}\label{eq:Gamma-reason-dist}
			\rho(\Gamma_2)\subset B^{r+1} \cap Y\quad\text{and}\quad{|\Gamma_2|>4d(k+3)}.
		\end{equation}
		
		Next, by applying \cref{lem:partitions}, condition \ref{prop:halves-pigeonhole},
		to $A = \rho(\Gamma_2)$ with weight function $f(v) \eqdef |\{x \in \Gamma_2 \mid \rho(x) = v\}|$,
		we obtain a subset of $\rho(\Gamma_2)$ that contains reasons for more than $2d(k+3)$ vertices of $\Gamma_2$ and is contained in one part of~$\P'_j$, for all~$j \ge i - 1$.
		Let $\Gamma_3$ be those vertices of~$\Gamma_2$. Therefore, $|\Gamma_3| > 2d(k+3)$ and
		\begin{equation}\label{eq:Gamma-reason-same-part}
			\text{$\rho(\Gamma_3)$ is contained in one part of~$\P'_j$, for 
				all~$j \ge i - 1$.}
		\end{equation}
		A final application of the pigeonhole principle allows us to
		restrict to a set $\Gamma \subseteq \Gamma_3$ with $|\Gamma|>d(k+3)$  such that
		\begin{equation*}
			\set{x\rho(x)\colon x\in \Gamma}\subseteq E(G)\quad\textrm{or}\quad \set{x\rho(x)\colon x\in \Gamma}\subseteq  E(\overline{G}).
		\end{equation*}
		By symmetry, for the remainder of the proof we assume that
		\begin{equation}
			\label{eq:Gamma-reason-positive}
			\set{x\rho(x)\colon x\in \Gamma}\subseteq E(G).
		\end{equation}
		
		\begin{claim}\label{cl:spread}
			For each~$j \ge i - 1$, there are at most~$d$ vertices from generation~$j$ in~$\Gamma$.
		\end{claim}
		\begin{claimproof}
			By \cref{eq:Gamma-same-part,eq:Gamma-reason-same-part}, there are parts $X',Y' \in \P'_j$ such that $\Gamma\subseteq X'$ and $\rho(\Gamma)\subseteq Y'$.
			Thus, all the vertices from generation~$j$ in~$\Gamma$ were added to create a dual set for $(X',Y')$ or for $(Y',X')$.
			Since~\cref{eq:Gamma-reason-positive} also implies that each $x \in \Gamma$ is adjacent to $\rho(x)$,
			they must in fact have been added to create a dominating set of~$Y'$ contained in~$X'$, as part of an $(X',Y')$-dual set.
			Since~$G$ has a duality of order~$d$, we added at most~$d$ vertices to make this dual set.
		\end{claimproof}
		By \cref{cl:spread}, the vertices of~$\Gamma$ have to be spread over more than~$k+3$ generations.
		Thus, there are steps $i - 1 \le j_1 < \dots < j_{k+4}$ such that for each~$\ell$, $\Gamma$ contains a vertex~$x_\ell$ of generation~$j_\ell$.
		Note that the insertion order $x_1 \qgt_\ins \dots \qgt_\ins x_{k+4}$ is reversed.
		Let~$r_\ell\eqdef\rho(x_\ell)$,
		and $X'_\ell,Y'_\ell$ be the parts of $\P'_{j_\ell}$ containing respectively~$\Gamma$ and~$\rho(\Gamma)$ (which exist due to~\cref{eq:Gamma-same-part,eq:Gamma-reason-same-part}).
		Let also~$X_\ell$ be the part of~$\P_{j_\ell}$ containing~$\Gamma$,
		so that~$X'_\ell$ is either equal to~$X_\ell$, or constitutes one half of~$X_\ell$ if the latter is medium.
		Finally, define $B^{r+2} \eqdef \Ball_{R_{i+1}}^{r+2}(v)$. 
		
		\begin{claim}\label{cl:zell}
			For each~$2<\ell \le k+4$, there exists some vertex $z_\ell \in (X_\ell \setminus X_{\ell-1}) \cap B^{r+2}$.
		\end{claim}
		\begin{claimproof}
			Recall that~$x_\ell r_\ell\in E(G)$ by~\cref{eq:Gamma-reason-positive}.
			This means that~$x_\ell$ was picked in generation~$j_\ell$ as part of a subset of $X'_\ell$ that  dominates~$Y'_\ell$.
			As~$r_{\ell-1}$ (and indeed all $r_1,\dots,r_{k+4}$) also belongs to~$Y'_\ell$,
			this dominating subset must also contain some vertex $z_\ell$ adjacent to~$r_{\ell-1}$ (see \cref{fig:zell}).
			Then:
			\begin{enumerate}
				\item As~$x_\ell$ was introduced to form a subset of $X'_\ell$ dominating~$Y'_\ell$,
				and~$z_\ell$ was used in the same dominating subset,
				$z_\ell$ was inserted earlier or in the same batch as~$x_\ell$.
				Therefore, we have $z_\ell \qle_\ins x_\ell \qlt_\ins x_{\ell-1}$.
				\item $z_\ell r_\ell\notin E(G)$.
				Indeed, since~$z_\ell,x_\ell\in X'_\ell \in \P'_{j_\ell}$ and $z_\ell \qle_\ins x_\ell$,
				this follows from \cref{clm:reason-minimal}.
				\item $z_\ell\in B^{r+2}$.
				Indeed, we have $z_\ell r_{\ell-1}\in E(G)$ and $z_\ell r_\ell\notin E(G)$.
				Since $r_\ell, r_{\ell-1}$ are both in the same part of~$\P_{i+1}$ by \cref{eq:Gamma-reason-same-part}, $z_\ell$ must be adjacent to one of them in~$R_{i+1}$.
				Furthermore, $r_\ell,r_{\ell-1}$ are both in~$B^{r+1}$ by \eqref{eq:Gamma-reason-dist}, thus $z_\ell$ is in~$B^{r+2}$.
				\item $z_\ell\notin X'_{\ell-1}$.
				Indeed, $z_\ell$ and~$x_{\ell-1}$ are both adjacent to~$r_{\ell-1}$ in~$G$,
				and $z_\ell \qlt_\ins x_{\ell-1}$, hence by \cref{clm:reason-minimal}, it cannot be that $z_\ell$ is in the same part $X'_{\ell-1} \in \P'_{j_{\ell-1}}$ as~$x_{\ell-1}$.
				\item Finally, $z_\ell\notin X_{\ell-1}$.
				Indeed, otherwise~$X_{\ell-1}$ has to be medium, with $X'_{\ell-1}$ being one half and~$z_\ell$ contained in the other half.
				In that case, \cref{lem:partitions}, condition \ref{prop:medium-halves}, implies that~$z_\ell$ is at distance more than~$2$ from~$X'_{\ell-1}$ in~$R_{j_{\ell - 1}+1}$.
				As $\ell > 2$, we have that
				% $i - 1 \leq j_1  < j_{\ell - 1} < j_{\ell - 1} + 1$, so
				$j_{\ell - 1} + 1 \geq i + 1$,
				so $z_\ell$ is at distance more than~$2$ from~$X'_{\ell-1}$ also in $R_{i + 1}$.
				% and a fortiori in~$R_{j_{\ell - 1}+1}$.
				But we have established that in~$R_{i + 1}$, $z_\ell$ is at distance~1 from either $r_\ell$ or $r_{\ell-1}$,
				and each of $r_\ell,r_{\ell-1}$ is at distance at most~1 from~$X'_{\ell-1}$ by \cref{eq:Gamma-reason-close}, a contradiction. \claimqedhere
			\end{enumerate}
		\end{claimproof}
		
    \begin{figure}[h]
      \centering
      \includegraphics[scale=1.6]{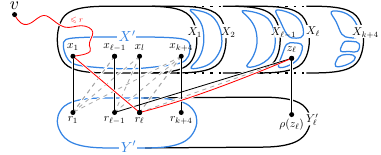}
      \caption{Illustration of \cref{cl:zell}. The partition $\P'_i$ is in \textcolor[rgb]{0.208, 0.518, 0.894}{blue}. Black edges are from $E(G)$, dashed \textcolor{gray}{gray} edges are from $E(\overline{G})$, \textcolor{red}{red} edges are from $R_{i+1}$. } \label{fig:zell}
    \end{figure}

		% Thus, for each~$\ell > 1$, we have $z_\ell \in (X_\ell \setminus X_{\ell-1}) \cap B^{r+2}$.
		\Cref{cl:zell} implies that $z_3,\dots,z_{k+4}\in B^{r+2}$ are~$k+2$ vertices in distinct parts of~$\P_{i-1}$, 
		which implies 
		$|B^{r+2}/\P_{i-1}|> k + 1$.
		This contradicts the assumption that the radius-$(r+2)$ width of the merge sequence is~$k$.
		This contradiction proves \cref{eq:mainbound}, and concludes the proof of \Cref{lem:duals}.
	\end{proof}
	We now wrap up the proof of \cref{lem:orderS}.
	Set $S_0\coloneqq V(G)$, and 
	for $i\in [m-1]$, set:
	$$S_{i}\coloneqq W_{i}\cup D_{i},$$
	where $W_i$ and $D_i$ are obtained from \Cref{lem:witnesses,lem:duals}, respectively.
	Then $D_i$ is a $(2,8)$-sample set for $(\P_i',R_{i+1})$, by \cref{lem:partitions}, condition~\ref{prop:sample}, and \cref{lem:duals}.
	By \cref{lem:witnesses,lem:duals} we obtain, for each $1<i<m$:
	$$\left|\Ball_{R_{i+1}}^{r}(v) \cap S_{i-1}\right|\le \O(dk^3)\quad\text{for $v\in V(G)$.}$$
	
	For $i = 1$ we know that $\Ball_{R_2}^{r}(v)$ intersects at most $k$ parts of~$\P_1$, which are singletons.
	Therefore, $|\Ball_{R_2}^{r}(v) \cap S_0| \leq k$, so the bound also holds. 
	This concludes the proof of \cref{lem:orderS}.
\end{proof}

We now complete the proof of \cref{thm:def-mw} using \cref{lem:orderS}. 

\begin{proof}[Proof of \cref{thm:def-mw}]
	Let $(\P_1, R_1), \dots , (\P_m , R_m)$ be a merge sequence for $G$ of radius-$(17r+3)$ width at most $k$.
  Without loss of generality we can assume that
  % $|\P_i| = |\P_{i + 1}| + 1$, i.e.
  $\P_1\prec \dots \prec \P_m$ is a maximal chain of partitions.
  Let $S_0 \supseteq S_1\supseteq \dots \supseteq S_{m-1}$ be subsets obtained from \cref{lem:orderS}. Consider a total order on $V(G)$ which first places $S_{m-1}$, followed by  $S_{m-2} \setminus S_{m-1}$, etc., up to $S_0\setminus S_1$, where each of those sets is ordered arbitrarily. %\sz{we still need to order $V(G)\setminus S_1$}
  Fix a set $S$ which forms a prefix of this order.
  Then there is some $i$
  with $S_i \subseteq S \subseteq S_{i-1}$. For  $v \in V(G)$, denote 
  \begin{equation*}
  S_v \eqdef \Ball_{R_{i+1}}^{17r+1}(v) \cap S.      
  \end{equation*}
  By \cref{lem:orderS} we have 
  \begin{align}\label{eq:bound_Sv}
  % |S_v| \le |\Ball_{R_{i+1}}^{17r+1}(v)\cap S_i|+|S_{i-1}\setminus S_i|\le  \O(dk^3).
  |S_v| \le |\Ball_{R_{i+1}}^{17r+1}(v)\cap S_{i-1}| \le  \O(dk^3).
  \end{align}
	Using similar arguments as in \cref{lem:mw-neigh-refinement}, we now show that \begin{equation} \label{eq:bound_Spart_Ri}
		|\Ball_{R_{i+1}}^{17r}(v) / \Nn_S| \le |S_v| + k\cdot \pi_G(|S_v|).
	\end{equation}
  % \karolina[inline]{
  First, observe that $\Ball_{R_{i+1}}^{17r}(v)$ contains at most $|S_v|$ elements of $S$, which form singleton parts of $\Nn_S$.
  It remains to bound the number of intersected parts of $\Nn_S$ within $V \setminus S$.
  
  Let $P \in \P_i$ be such that $\Ball_{R_{i+1}}^{17r}(v) \cap P \neq \emptyset$.
  Note that the number of such parts is at most~$k$. Let $A \eqdef (\Ball_{R_{i+1}}^{17r}(v) \cap P)\setminus S$.
  Observe that for each $u \in S \setminus S_v$ there is no resolved pair between $u$ and $A$ in $R_{i + 1}$ (otherwise $u$ would be in $S_v$).
  Since $\P_i$ is homogeneous modulo~$R_{i + 1}$, $A$ is homogeneous to $u$.
  Therefore, |$A/ \Nn_S| = |A/ \Nn_{S_v}| \leq \pi_G(|S_v|)$.
  This proves \cref{eq:bound_Spart_Ri}.
  % Therefore, the number of parts of $\Nn_S$ intersected by $A$ is bounded by the number of different neighbourhood types on $S_v$.
  % So we get $|\Ball_{R_{i+1}}^{17r}(v) / \Nn_S| \le |S_v| + k \cdot \pi_G(|S_v|)$.

	Finally, since $S$ is a $(2,8)$-sample set of $(\P_i', R_{i+1})$, by \cref{lem:incremental}, there exists an $S$-flip $G'$ of $G$ such that $$\Ball_{G'}^r(v) \subseteq \Ball_{R_{i+1}}^{17r}(v).$$ Now, the first inequality in the conclusion of the theorem follows by combining this with \cref{eq:bound_Spart_Ri,eq:bound_Sv}, while the second inequality follows from the Sauer-Shelah-Perles lemma, \Cref{sauer_shelah_lemma}.
\end{proof}

\section{Proofs of the main results}
\label{sec:proofs}
In this section, we wrap up the proofs of the main results stated in the introduction, namely \Cref{thm:intro,thm:cw,thm:almost-bounded}.

\paragraph{Bounded merge-width}
\Cref{thm:intro} states that bounded merge-wid  th, bounded partition merge-width, bounded transient merge-width, and bounded definable merge-width are all equivalent properties of graph classes.
As noted earlier, this follows from 
 \cref{thm:main1}, \cref{lem:dmw-pmw,lem:pmw-tmw,lem:tmw-mw}, and \cref{thm:main2}.

\paragraph{Infinite radius}
Merge-width, and its transient, partition, and definable variants each have an associated limit parameter,
denoted $\mw_\infty, \tmw_\infty,\pmw_\infty$, and $\dmw_\infty$  respectively, and defined by the same definitions as the finite-radius variants, but
interpreting $\Ball^\infty(v)$ as $\bigcup_{r\ge 1} \Ball^r(v)$ in each of the definitions. Equivalently, $\mw_\infty(G)=\mw_r(G)$ where $r\ge |V(G)|-1$, and similarly for the remaining variants.
Therefore, by \Cref{lem:tmw-mw} and \Cref{lem:pmw-tmw}  we have:
$$\pmw_\infty(G)\le \tmw_\infty(G)\le \mw_\infty(G).$$
\Cref{lem:dmw-pmw} implies
$$\pmw_\infty(G)\le 2\cdot \dmw_\infty(G)+1.$$
Furthermore, by \cref{thm:main1},
 $\mw_\infty(G)$ is bounded in terms of $\pmw_\infty(G)$, while
by \Cref{thm:main2},
$\dmw_\infty(G)$ is bounded in terms of $\mw_\infty(G)$.
 Thus, all those parameters are functionally equivalent to each other.
 Since $\mw_\infty(G)$ is functionally equivalent to clique-width \cite[Thm.~7.1]{merge-width}, altogether this proves \Cref{thm:cw}.

 \paragraph{Almost bounded merge-width}
%  We turn to the proof of \Cref{thm:almost-bounded}.
 Recall that a graph class $\CC$ has \emph{almost bounded merge-width}
 \cite{merge-width}
 if for every fixed $r\in\N$ and $\varepsilon>0$,
 for sufficiently large $n$ we have that $\mw_r(G)\le n^\varepsilon$ for every $n$-vertex graph $G\in \CC$.
More concisely, for fixed $r$ we have that  $\mw_r(G)\le |V(G)|^{o(1)}$ holds for all $G\in\CC$.
Analogously, we define \emph{almost bounded partition merge-width}, 
\emph{almost bounded definable merge-width}, and 
\emph{almost bounded transient merge-width}.
\Cref{thm:almost-bounded} states that for a hereditary graph class $\CC$, those four notions are equivalent.

Towards a proof of \Cref{thm:almost-bounded},
 first observe that for classes $\CC$ of VC-dimension bounded by a fixed constant $d$, the four conditions are all equivalent.
 This again follows by 
 \cref{thm:main1}, \cref{lem:dmw-pmw,lem:pmw-tmw,lem:tmw-mw}, and \cref{thm:main2},
 because each of the bounds relating one notion of merge-width with another in those statements is polynomial, for graphs $G$ of VC-dimension bounded by a fixed constant $d$ 
 (recall that $\pi_G(n)\le \O(n^d)$ by \Cref{sauer_shelah_lemma}),
 and we have that $\textit{poly}(n^{o(1)})=n^{o(1)}$ for any fixed polynomial $\textit{poly}$.

 To conclude the proof of \Cref{thm:almost-bounded},
it is therefore enough to prove that if $\CC$ is a  \emph{hereditary}
graph class $\CC$ satisfying either of the four conditions,
then $\CC$ has VC-dimension bounded by a constant. 
Since almost bounded partition merge-width is the weakest of the four properties (it is implied by each other property), it is enough to prove this for hereditary classes of almost bounded partition-merge width.
To this end, we prove the following lemma, which is an analogue of \cite[Cor.~5.26]{flip-width}.

\begin{restatable}{lemma}{VCdim}\label{lem:VCdim}
  Let $G$ be a graph with $\VC(G)\ge d$.
  Then $G$ contains an induced subgraph $H$ 
  with at most $2d$ vertices and
  $$\pmw_2(H)\ge \Omega(d).$$
\end{restatable}

\begin{proof}
For  $m\in\N$, let $\mathbb{Z}_2^m$ be the $m$-dimensional vector space over the two-element field $\mathbb{Z}_2$.
We define $H_m$ to be the bipartite graph with domain consisting of two copies of $\mathbb{Z}_2^m$, denoted $V$ and~$V^*$, such that $u\in V, v\in V^*$ are adjacent if and only if $0 = u\cdot v \eqdef \sum_{i=1}^n u_i\cdot v_i$ (with addition considered in $\mathbb Z_2$, that is, modulo 2).

Fix $d\in\N$. 
Let $G$ be a graph with $\VC(G) \ge d$,
and $m$ be largest so that $2^m\leq d$. 
It is easy to see (see proof of \cite[Lem.~5.25]{flip-width}) that $V(G)$ contains two (possibly overlapping) subsets $A,A^*$ of size $2^m$ each, such that 
the bipartite graph $G[A,A^*]$ 
with vertex set $A\uplus A^*$ and edge set $\set{uv\in E(G)\colon u\in A, v\in A^*}$
is isomorphic to $H_m$.
Via this isomorphism, we identify $A$ with the two copies of $\mathbb{Z}_2^m$, and treat the elements of $A,A^*$ as vectors in $\mathbb{Z}_2^m$.

We argue that $H\coloneqq G[A\cup A^*]$ has $\pmw_2(H)\geq \frac{d}{2\cdot 16\cdot 17}$.
If $d< 32$ this is trivial, so we assume that $d\ge 32$ and we have that $|A|=|A^*|=2^m\ge 32$.

Let $\P_1 \prec P_2\prec \dots, \prec \P_n$ be any sequence of partitions of $V(H)$ such that $\P_1$ is the partition into singletons and $\P_n$ contains one part.
Let $t<n$ be smallest such that $\P_{t+1}$ contains a part~$P$ of size at least $18$. Without loss of generality, this part contains at least $9$ elements from $A^*$. 
Among them are four linearly independent vectors $w_1,w_2,w_3,w_4$, as a three-dimensional space over $\mathbb Z_2$ has size at most $2^3=8$.

For $\textbf{b} \in \{0,1\}^4$, 
  denote $$W_{\textbf{b}}  \eqdef \{x\in A\mid x\cdot w_i = b_i\quad\text{for } i\in [4]\}.$$
\begin{claim}For all $\textbf{b}\in \{0,1\}^4$, we have $|W_{\textbf{b}}|=2^{m-4}.$

\end{claim}
\begin{claimproof}This is elementary linear algebra, as we have four linearly independent equations in $m$ variables. For the sake of completeness, we provide a proof.

   Let $M$ be the $4\times m$ matrix whose rows are vectors $w_1,\dots.w_4$.
  As $w_i$'s are linearly independent, $M$ has rank $4$ and the linear map $f: \mathbb{Z}_2^n \rightarrow \mathbb{Z}_2^4$, given by $x\mapsto Mx$, is surjective.
  By Rank-Nullity, the kernel of $f$ has dimension $m-4$, so $|f^{-1}(\mathbf 0)|=2^{m-4}$, where $\mathbf 0\in \mathbb{Z}_2^4$ is the zero vector.
  The set $W_{\mathbf{b}}=f^{-1}(\mathbf{b})$ is a translate of $f^{-1}(\mathbf 0)$, as 
   $f^{-1}(\mathbf{b})=f^{-1}(\mathbf 0)+v$ for any $v\in f^{-1}(\mathbf{b})$. Thus  $|W_{\mathbf{b}}|=2^{m-4}$.
  % For $\mathbf{b}\in \{0,1\}^4$, take any $v\in f^{-1}(\mathbf{b})$. 
  % Consider the bijective map $g_\mathbf{b}:\mathbb{Z}_2^m\rightarrow \mathbb{Z}_2^m$ given by $x\mapsto x+ v$. 
  % Then $f^{-1}(\mathbf{b})= g_{\mathbf{b}}(f^{-1}(0))$ is of the same size as $f^{-1}(0)$, hence $|W_{\mathbf{b}}|=2^{m-4}$.
\end{claimproof}
Let $\textbf{b}_1=(1,1,0,0), \textbf{b}_2=(1,0,1,0)$.
Fix $v_1 \in W_{\textbf{b}_1}$ and consider any $v_2\in W_{\textbf{b}_2}$.
Let $S_1, S_2$ be the parts of $\P_{t+1}$ containing $v_1,v_2$ respectively. 
Recall that $w_1,\ldots,w_4$ are all in the same part $P$ of $\P_{t+1}$.
In every $\P_{t+1}$-flip $H'$ of $H$,
$v_1$ is adjacent to either $\set{w_1,w_2}$ or to $\set{w_3,w_4}$, while $v_2$ is adjacent to either $\set{w_1,w_3}$ or to $\set{w_2,w_4}$. It follows that 
 there is some $w_i$ that is adjacent to both $v_1$ and $v_2$ in $H'$.
Therefore, $v_2 \in \Ball_{P_{t+1}}^2(v_1)$, and so $W_{\textbf{b}_2}\subset \Ball_{P_{t+1}}^2(v_1)$.

Finally, since each part of $\P_t$ has at most 17 elements, $W_{\textbf{b}_2}$ intersects at least $2^{m-4}/17$ parts,
and hence $$|\Ball_{P_{t+1}}^2(v_1)/\P_t|\ge \frac{2^{m-4}}{17}.$$
As $2^m\geq d/2$, this proves that $\pmw_2(H)\geq \frac{d}{2\cdot 16\cdot 17}$.
\end{proof}

\begin{corollary}\label{cor:almost-bounded-VC}
  Let $\CC$ be a hereditary graph class with $$\pmw_2(G)\le o(|V(G)|)\qquad\text{for all $G\in\CC$}.$$ Then $\VC(\CC)<\infty$. In particular, $\VC(\CC)<\infty$ for every hereditary graph class $\CC$ of almost bounded partition merge-width.
\end{corollary}
With the earlier remarks, this implies \Cref{thm:almost-bounded}.

\section{Applications}\label{sec:apps}
\subsection{Sparse quotients, quasi-isometry, and neighbourhood covers}\label{sec:quasi-iso}
As applications of \cref{thm:main1}, 
we show that graphs of bounded merge-width admit partitions such that the quotient graph has bounded strong colouring numbers, and each part has small weak diameter. This will yield two applications: quasi-isometries with graphs of bounded strong colouring numbers, and sparse neighbourhood covers. We remark that the arguments presented in this section have much simpler counterparts in the more restricted setting of graphs of bounded twin-width; they moreover yield better bounds. Hence, for the convenience of the readers interested in twin-width, we provide an exposition in \cref{sec:tww-app}.

\medskip
For a graph~$G$ and a partition~$\P$ of~$V(G)$, we define the \emph{quotient graph} $G/\P$ to be the graph obtained from~$G$ by contracting each part of~$\P$ into a single vertex.
That is, $V(G/\P) = \P$, and for any distinct $P,Q \in \P$, $PQ \in E(G/\P)$ if and only if there exist $x \in P$, $y \in Q$ with $xy \in E(G)$.
A set $X\subset V(G)$ has \emph{weak diameter} $r$ in $G$ if $\dist_G(x,y) \le r$ for all $x,y \in X$.
% The following lemma shows that graphs of bounded merge-width admit partitions into parts of bounded diameter, yielding a quotient graph with bounded strong colouring numbers.
\begin{lemma}
  \label{lem:quasi-iso}Fix $r\in\N$.
  Let~$G$ be a graph and $k=\mw_{18r+1}(G)$.
  % Fix $\ell = k \cdot \pi_G(k)$.
  Then 
  there is a partition~$\P$ of~$V(G)$ satisfying the following:
  \begin{enumerate}
    \item \label{item:scol} $\scol_r(G/\P) \le k \cdot \pi_G(k)$, and
    \item \label{item:quasi-isometry} 
    each part of $\P$ has weak diameter at most $12$ in $G$.        
  \end{enumerate}  
\end{lemma}
\begin{proof}
  By \cref{thm:main1}, there is a merge sequence $$(\P_1,R_1), \dots, (\P_m,R_m)$$ with radius-$3r$ width at most~$\ell\coloneqq k \cdot \pi_G(k)$ which is 6-firm, i.e.\ $\dist_G(x,y) \le 6$ for all $xy \in R_m$.
  Since $\P_1$ is the partition into singletons, we may assume that $R_1=\emptyset$.

  Say that a part~$P \in \P_i$ is \emph{unresolved} if there exists some unresolved edge~$xy \in E(G) \setminus R_i$ with $x \in P$.

  Suppose first that the unique part $V(G)$ of $\P_m$ is unresolved (meaning that $E(G) \not\subseteq R_m$).
  Since there is an unresolved edge, every non-edge must be resolved (in~$R_m$).
  Then for every $x,y \in V(G)$,
  if~$xy\notin E(G)$ then $\dist_G(x,y) \le 6$ by assumption on the merge sequence,
  and if $x,y \in E(G)$ then $\dist_G(x,y) \le 1$. Hence, $G$ has diameter at most~6, and we may take $\cal P=\{V(G)\}$, which satisfies the conclusion of the lemma.
  Henceforth we assume this case does not occur.

  Say furthermore that an unresolved part $P$ is \emph{maximally unresolved} with \emph{index}~$i$ if~$P \in \P_i$ is unresolved,
  but for every $j>i$, the part of~$\P_{j}$ containing~$P$ is resolved.
  In particular, all edges of~$G$ incident to $P_j$ are contained in $R_{j}$.
By the previous assumption, any maximally unresolved part has index strictly less than~$m$.

  We assume that $G$ has no isolated vertex, as these can be treated in a trivial manner, by placing each isolated vertex in its own part of the partition.
  Let~$\P$ be the set of all maximally unresolved parts across $\P_1,\dots,\P_m$.
  Note that since~$G$ has no isolated vertex and $R_1=\emptyset$, each $v\in V(G)$ belongs to some unresolved part.
  It follows that~$\P$ is a partition of~$V(G)$.

  \begin{claim}\label{clm:max-unres-diam}
    Let~$P$ be a maximally unresolved part with index~$p<m$.
    Then~$P$ has weak diameter~$2$ in~$(V(G),R_{p+1})$, and weak diameter at most $12$ in $G$.
  \end{claim}
  \begin{claimproof}
    Since~$P$ is unresolved, there is some $xy \in E(G) \setminus R_p$ with~$x \in P$.
    Let~$Q \in \P_p$ be the part containing~$y$ (possibly $P=Q$).
    Fix an arbitrary $y'\in Q$.
    We will argue that  $x'y' \in R_{p+1}$ for any $x'\in P$.
    This implies that $\dist_G(x',y')\le 6$ by firmness, and proves the claim.

    There are two cases. If~$x'y'$ is an edge, then it must be in~$R_{p+1}$, as otherwise~$P$ would not be maximally unresolved.
    If on the other hand~$x'y'$ is a non-edge, then it must be in~$R_p$,
    as otherwise there would be both the unresolved edge~$xy$ and the unresolved non-edge~$x'y'$ between~$P$ and $Q$, contradicting that~$\P_p$ is homogeneous modulo~$R_p$. As $R_p\subseteq R_{p+1}$, this implies $x'y' \in R_{p+1}$.
  \end{claimproof}

  Order~$\P$ by reversed indices, breaking ties arbitrarily.
  That is, pick a total ordering~$<$ of~$\P$ such that if~$P,Q$ have indices~$i,j$ respectively, then $P<Q$ implies $i \ge j$.
  This yields a total order~$<$ on the vertices of the quotient graph~$G/\P$.
  \begin{claim}\label{clm:wcol-MaxUnrPart}
    For every $P\in \P$, in $G/\P$ we have $|\sreach_{r}^<(P)|\le \ell$.
  \end{claim}
  \begin{claimproof}
        Let~$p$ be the index of~$P$, and recall that we may assume $p<m$. Let $v_0\in P$ be any vertex. We argue that each part  $S\in\P$ which is strongly~$r$-reachable from~$P$ in~$G/\P$ is also reachable from $v_0$ in~$R_{p+1}$ via a path of length at most~$3r$. This will imply that there are at most~$\ell$ such parts~$S$.

        Since $S$ is strongly $r$-reachable from $P$, there is a path $P=P_0,P_1,\dots,P_t=S$ in~$G/\P$ with $t \le r$.
        Let $p=p_0,p_1,\dots,p_{t-1}$ be the indices of $P_0,P_1,\dots,P_{t-1}$.
        We have $S < P <  P_{i}$ for all~$i$, and thus $s \ge p \ge p_i$.
        Using \cref{clm:max-unres-diam} and the monotonicity of~$(R_j)_{j \in [m]}$, all parts $P,P_i$ have weak diameter at most~$2$ in~$(V,R_{p+1})$.
        Moreover, as these parts are maximally unresolved, all the edges of~$G$ that have an endpoint in one of $P,P_i$ are already present in $R_{p+1}$.
        Therefore, we can lift the path $P_0,\ldots,P_t$ to a path in $(V,R_{p+1})$ of length at most $3r$ in $G$, which starts in $v_0$ and ends in some $w\in P_t$
        (this is easily proved by induction on $t$).
        Hence, there is $u \in S$ such that $\dist_{R_{p+1}}(v_0,u) \le 3r$.
        Note that any two different $S,S'<P$ are supersets of different parts of $\P_p$ and so there are at most $\ell$ parts within~$\sreach_{r}^<(P)$ in~$G/\P$.
  \end{claimproof}
  \Cref{clm:max-unres-diam} and \Cref{clm:wcol-MaxUnrPart} together yield the lemma.
\end{proof}

Let $\cal P$ be the partition as in \Cref{lem:quasi-iso}. For $x\in V(G)$, we denote by $[x]$ the part of $\cal P$ containing $x$.
Then for any~$x,y \in G$,
if $[x]$ and $[y]$ are at distance at most $r$ in $G/\P$, then $\dist_G(x,y) \le r+12(r+1)=13r+12$.
This is proved by induction on $r$, using the fact that any two vertices in the same part of $\P$ are at distance at most $12$ in $G$.
Hence, we get that for all $x,y\in V(G)$,
\begin{equation}\label{eq:quasi-isom}
  \dist_{G/\P}([x],[y]) \le \dist_G(x,y) \le 13 \cdot \dist_{G/\P}([x],[y]) + 12.
 \end{equation}

\paragraph{Quasi-isometry} 
In the language of coarse geometry (see e.g. \cite[Sec.~2.1]{coarse-graphs}), \cref{eq:quasi-isom} implies that $G$ and $G/\P$ are \emph{$(13,12)$-quasi-isometric}. Formally, quasi-isometries are defined as follows:
\begin{definition}
	Let $(X,\delta_X)$ and $(Y,\delta_Y)$ be metric spaces. For constants $a\geq 1$ and $b\geq 0$, an {\em{$(a,b)$-quasi-isometry}} from $(X,d_X)$ to $(Y,d_Y)$ is mapping $\phi\colon X\to Y$ satisfying the following:
	\begin{itemize}
		\item $\frac{1}{a}\cdot d_X(x,x')-b\leq d_Y(\phi(x),\phi(x'))\leq a\cdot d_X(x,x')+b$, for all $x,x'\in X$; and
		\item for each $y\in Y$ there exists $x\in X$ such that $d_Y(y,\phi(x))\leq b$.
	\end{itemize}
\end{definition}
So intuitively, a quasi-isometry is a mapping that preserves distances up to fixed multiplicative and additive factors ($a$ and $b$, respectively) and whose image is roughly dense in the co-domain. The definition can be applied to graphs by taking the distance metric on the vertex set.

%Intuitively, the distances in $G$ and $G/\P$ are similar up to fixed multiplicative and additive constants (here, $3$ and $2$, respectively). 
% In the corner case when $G$ has diameter at most~$6$, $G$ is trivially $(1,6)$-quasi-isometric to a single vertex. 
We thus obtain the following statement.
% In any case, we get the following statement, which roughly says that graphs of bounded merge-width are quasi-isometric to graphs with bounded strong colouring numbers.

\begin{theorem}\label{thm:quasi-isom}
  Fix $r\in\N$.
  Let $G$ be a graph and $k=\mw_{18r+1}(G)$.
  Then $G$
   is $(13,12)$-quasi-isometric to some $H$ with \mbox{$\scol_r(H)\le k \cdot \pi_G(k)$}.
\end{theorem}

Furthermore, if $\CC$ is a class with universally bounded merge-width,
then the quotients form a class with bounded expansion. 

\quasiIsoBE*

\begin{proof}
	Observe that in the proof of \cref{thm:main1}, if we start with a merge sequence with radius-$r$ width at most $f(r)$ for every~$r$,
	then we obtain a merge sequence with radius-$r$ width bounded by $f'(r)$ for every~$r$, for some function $f'$ depending only on~$f$. Moreover, the sequence has the property that $\dist_G(x,y) \le 6$ whenever \mbox{$xy \in R_m$}.
	Then, the  proof of \cref{lem:quasi-iso} yields a partition~$\P_G$ such that $\scol_r(G/\P_G)$ is bounded by $f''(r)$ for every~$r$, for some function $f''$ depending only on $f$,
	so that $G$ is $(13,12)$-quasi-isometric to $G/\P_G$. The class $\D = \{G/\P_G : G \in \CC\}$ satisfies $\scol_r(\D) \le f''(r)$ for all $r$, and is thus a class of bounded expansion.
\end{proof}

% The following corollary is immediate.
% \begin{corollary}\label{cor:quasi-isom}
%   Let $\cal P$ be the partition as in \Cref{lem:quasi-iso}. Then
%   for any~$x,y \in G$,
%       \[ \dist_{G/\P}([x],[y]) \le \dist_G(x,y) \le 3 \cdot \dist_{G/\P}([x],[y]) + 2, \]
%       where~$[x]$ denotes the part of $\P$ containing~$x$.
% \end{corollary}

\paragraph{Neighbourhood covers} As a further application of \cref{lem:quasi-iso}, we show that graphs of bounded merge-width admit sparse neighbourhood covers of small radius.  
\begin{definition}\label{def:nbhdCov}
    A \emph{neighbourhood cover} of a graph~$G$ of~\emph{overlap}~$k$ and \emph{radius}~$t$ 
    is a family $\K$ of non-empty subsets of~$V(G)$ such~that
    \begin{enumerate}
        \item \label{it:nbhdCov_i} for every~$v\in V(G)$ there is~$C\in \K$ with~$N(v) \subseteq C$;
        \item \label{it:nbhdCov_ii} for every~$C\in\K$ there is~$v \in V(G)$ with~$C\subseteq\Ball_G^t(v)$; and
        \item \label{it:nbhdCov_iii} for every~$v\in V(G)$, the set~$\{C \in \K \mid v\in C\}$ has size at most $k$.
    \end{enumerate}
\end{definition}
All graphs with linear neighbourhood complexity (in particular, graphs of bounded radius-$2$ merge-width, by \cref{thm:nbd_complexity}) admit neighbourhood covers with bounded radius and overlap $\O(\log n)$, by the result of Dreier et al.~\cite{demmpt}. \Cref{thm:nbhdCov} below decreases the overlap to a constant, in classes of bounded merge-width. We note that we are not aware of any previous construction of such neighbourhood covers even for the more restrictive case of classes of bounded twin-width.

We use the following lemma, which is a special case of \cite[Lem.~6.10]{gks} (note that \cite{gks} use a stronger notion of neighbourhood cover, requiring that $G[C]$ has radius ${\le}t$ for each $C\in\K$).
\begin{lemma}[\cite{gks}]\label{lem:NbdCovWcol}
  Let $G$ be a graph and $p=\wcol_2(G)$. Then $G$ admits a neighbourhood cover of overlap~$p$ and radius~$2$.
\end{lemma}
% \begin{proof}[Proof sketch]
% For $v\in V(G)$, let $C_v = \{u \in V(G) : v \in \wreach_2(u)\}$ and define $\F = \{C_v : v \in V(G)\}$. It is easily checked that $\F$ satisfies the required conditions.
% \end{proof}

Combining \cref{lem:quasi-iso,lem:NbdCovWcol}, we obtain the following.
\begin{theorem}\label{thm:nbhdCov}
    Let $G$ be a graph and  $k=\mw_{37}(G)$. Further, let $\ell = k\cdot\pi_G(k)$.
    Then $G$ admits a neighbourhood cover of overlap~$\ell^2$ and radius~$38$.
\end{theorem}
\begin{proof}
  By \cref{lem:quasi-iso} there exists partition~$\P$ of~$V(G)$ with the following properties: each $P\in\P$ has weak diameter at most~$12$ in~$G$ and there is an order $<$ on $\P$ witnessing $\scol_2(G/\P) \le \ell$. By \cref{lem:col-nb-equiv}, such an order also witnesses $\wcol_2(G/\P)\le \ell^2$.

By \Cref{lem:NbdCovWcol}, graph $G/\P$ admits a neighbourhood cover $\K$ of overlap at most~$\wcol_2(G/\P) \le \ell^2$ and radius~$2$.
Then $\{\bigcup C \colon C \in \K\}$ is a neighbourhood cover of $G$ of overlap at most $\ell^2$ and radius at most $13\cdot 2 + 12=38$,
by~\cref{eq:quasi-isom}.
\end{proof} 

    % We now define the neighbourhood cover for $G$. For each $P\in\P$, let
    % \[
    % A_P=\bigcup\{Q\in \P : P \in \wreach_2^<(Q)\}
    % \]
    % and define
    % \[
    % \A = \left\{A_P: P \in \P\right\}
    % .\]
    
    % Fix $v\in V$ and let $P\in \P$ be the part containing $v$. Let~$Q\in \P$ be the smallest, wrt.~$<$, part containing a neighbour of~$v$. 
    % Then any part containing a neighbour of~$v$ weakly~$2$-reaches~$Q$ via a path through~$P$, hence~$N(v)\subseteq A_Q$. 
    % This shows \cref{it:nbhdCov_i} of \cref{def:nbhdCov}.
    % Furthermore, if~$v\in A_Q$ for some $Q\in \P$, then by construction~$Q\in \text{WReach}_2^<(P)$, so \cref{it:nbhdCov_iii} holds with overlap~$1+ 2(l-1)^2$.
    % Finally, let $u\in A_P$, and $Q\in \P$ the part containing $u$. By definition, $P$ is $2$-reachable from $Q$ in $G/\P$. Since every part of $\P$ has diameter at most~$12$ in~$G$, it follows that $\dist_G(u,v) \le 3\cdot 12+2$. Therefore $A_P \subseteq \Ball^{38}_G(v)$ and \cref{it:nbhdCov_ii} holds.

\subsection{Approximating merge-width}\label{sec:apxmw}
As an application of \cref{thm:main1,thm:main2}, we obtain a single-exponential time approximation algorithm for merge-width.%
\apxMW*
\cref{thm:approx-intro} follows from the next lemma.
\begin{restatable}{lemma}{approxmw}\label{lem:approx-mw}
  There is an algorithm that, given an $n$-vertex graph~$G$ and $r\in\N$, computes in time $n^{\O(1)}\cdot 2^n$
  a transient merge sequence with radius-$r$ width at most
  \[ (2k+1) \cdot \pi_G(2k+1) \qquad \text{where } k = \dmw_{6r+1}(G). \]
\end{restatable}

\begin{proof}
	Let $G$ be an $n$-vertex graph and $r\in\N$, and fix an arbitrary ordering of~$V(G)$.
	
	For any partition~$\P$ of~$V(G)$, let~$X_\P$ be the transversal of~$\P$ obtained by selecting the first element of each part according to the chosen ordering of~$V(G)$,
	and call $\P' \eqdef \P \wedge X_\P$ the corresponding refinement.
	Note that when~$\P$ refines~$\Q$, we have $X_\P \supseteq X_\Q$, and thus~$\P'$ refines~$\Q'$.
	Finally, let~$G_\P$ be the $\P'$-flip of~$G$ given by \cref{lem:informal} such that
	\begin{equation}\label{eq:balls-flip-metric}
		\Ball^r_{G_\P}(v) \subseteq \Ball^{6r}_\P(v) \quad \text{for any $v \in V(G)$.}
	\end{equation}
	Given~$\P$, it is trivial to compute~$X_\P$ and~$\P'$ in polynomial time, and recall that~$G_\P$ can also be found in polynomial time.
	We shall assume that some tie-breaking rule is used in the construction of~$G_\P$ so that it is uniquely defined.
	
	Next, for any subset $S \in V(G)$, call~$\Nn_S$ the partition into atomic types over~$S$.
	Given~$S$, one can again compute~$\Nn_S$, $\Nn'_S$, and~$G_{\Nn_S}$ in polynomial time.
	
	Consider now the directed graph $\cal K$ with vertex set $2^{V(G)}$ and edges  $(S,S\cup \{v\})$ for each $S\subsett V(G)$ and $v\in V(G)\setminus S$.
	This graph has $2^n$ vertices and less than $n2^{n}$ edges.
	We assign to each edge $(S,T)$ a weight $f(S,T)$ defined by
	\[ f(S,T) \eqdef \max_{v\in V(G)} |\Ball_{G_{\Nn_s}}^{r}(v)/\Nn'_{T}|. \]
	This can be computed in polynomial time by considering each $v\in V(G)$ in turn.
	In total, the weight function $f$ can be computed in time $n^{\O(1)}\cdot 2^n$.
	
	Let $\ell$ be the least integer such that there is a directed path in~$\cal K$ from $\emptyset$ to $V(G)$ with maximum edge weight at most $\ell$. We can determine $\ell$ in time $n^{\O(1)} \cdot 2^n$ using e.g.\ breadth-first search.
	
	\medskip
	We first verify
	\[ \ell \le (2k+1) \cdot \pi_G(2k+1) \quad \text{where } k \eqdef \dmw_{6r+1}(G). \]
	By definition, there is a total order on $V(G)$ such that for each prefix $S$ of this order and each $v\in V(G)$,
	\[ \left|\Ball_{\Nn_S}^{6r+1}(v)/ \Nn_{S}\right| \le  k. \]
	For two consecutive prefixes $S\subsett T$ we have that $\Nn_{T}$ is a refinement of 
	$\Nn_S$ such that each part of $\Nn_S$ is the union of at most two parts of $\Nn_{T}$, plus possibly the singleton $T \setminus S$.
	It follows that
	\[ \left|\Ball_{\Nn_S}^{6r+1}(v)/ \Nn_{T}\right| \le  2k+1. \]
	Next, applying \cref{cor:mw-transversal-refinement} to~$\Nn_T$, $\Nn_S$, and the transversal~$X_{\Nn_T}$, we obtain the following bound for the refined partition~$\Nn'_T$:
	\[ \left|\Ball_{\Nn_S}^{6r}(v)/ \Nn'_{T}\right| \le  (2k+1) \cdot \pi_G(2k+1). \]
	Finally, by \eqref{eq:balls-flip-metric}, we find in the graph~$G_{\Nn_S}$:
	\[ \left|\Ball_{G_{\Nn_S}}^{r}(v)/ \Nn'_{T}\right| \le  (2k+1) \cdot \pi_G(2k+1). \]
	This precisely proves $f(S,T) \le (2k+1) \cdot \pi_G(2k+1)$.
	Following the ordering of~$V(G)$ given by definable merge-width,
	we thus obtain a directed path in $\cal K$ from $\emptyset$ to $V(G)$ such that each edge on this path has weight at most $(2k+1) \cdot \pi_G(2k+1)$.
	
	\medskip
	Consider now a path in~$\cal K$ with maximum edge weight at most~$\ell$, i.e.\ a sequence
	$$\emptyset=S_0\subsett S_1\subsett \dots \subsett S_n=V(G)$$
	such that for each $i\in [n]$, we have $f(S_{i-1},S_{i})\le \ell$.
	Let us show that from this sequence, we can in polynomial time construct a transient merge sequence with radius-$r$ width at most $\ell \cdot \pi_G(\ell)$, proving the result.
	
	Let $\P_i \eqdef \Nn_{S_{n-i}}$ for each $0 \le i \le n$, and $\P_i' \eqdef \Nn'_{S_{n-i}}$ the corresponding refined partition.
	Clearly~$\P_i$ refines~$\P_{i+1}$, and recall that this implies that~$\P'_i$ refines~$\P'_{i+1}$.
	Also, it is simple to check that $\P_0 = \P'_0$ is the partition into singletons,
	and $\P_n = \P'_n$ has a single part~$V(G)$.
	Finally, let $G_i \eqdef G_{\P_i} = G_{\Nn_{S_{n-i}}}$ be the chosen $\P'_i$-flip of~$G$.
	
	Now by assumption we have $f(S_{n-i},S_{n-i+1}) \le \ell$, meaning that for any $v \in V(G)$,
	\begin{equation*}%\label{eq:approx-mw-2}
		|\Ball_{G_{i}}^{r}(v)/ \P'_{i-1}| \le  \ell.
	\end{equation*}
	Thus $(\P'_0,G_0),\dots,(\P'_n,G_n)$ is a transient merge sequence with radius-$r$ width at most~$\ell$.
\end{proof}

\begin{proof}[Proof of \cref{thm:approx-intro}]
  Let~$d$ denote the VC-dimension of~$G$, and $\ell = \mw_{612r+122}(G)$.
  By \cref{thm:main2}, we have
  \[ k_1 \eqdef \dmw_{36r+7}(G) \le (\ell d)^{\O(d)}. \]
  Now \cref{lem:approx-mw}, applied to~$G$ with radius~$6r+1$, gives in time~$n^{\O(1)} \cdot 2^n$ a transient merge sequence with radius at most
  \[ k_2 \eqdef (2k_1+1) \cdot \pi_G(2k_1+1). \]
  Finally, by \cref{thm:main1}, this transient merge sequence can be turned into a (standard) merge sequence of radius-$r$ width at most
  \[ k_3 \eqdef k_2 \cdot \pi_G(k_2). \]
  This can be done in polynomial time.

  Combining the inequalities on~$\ell,k_1,k_2,k_3$, and using the bound $\pi_G(t) \le t^{\O(d)}$ (\cref{sauer_shelah_lemma}),
  we find that
  \[ k_3 \le (d \ell)^{\O(d)} \le \ell^{\O(\ell)}, \]
  where the last inequality uses $d \le \O(\ell)$ (\cref{thm:vc_mw1}).
  Thus we have constructed a merge sequence with radius-$r$ width at most $\ell^{\O(\ell)}$ for $\ell =  \mw_{612r+122}(G) \le \mw_{734r}(G)$.
\end{proof}

\subsection{Adjacency labelling schemes}\label{sec:labscheme}

In this section we give an adjacency labelling scheme for graphs of bounded radius-$1$ short merge-width. We start by recalling the formal definition of an adjacency labelling scheme.

%\begin{definition}
%	The {\em{valency}} of a merge sequence $(\P_1,\ldots,\P_m)$ of a graph $G$ is the smallest integer $\Delta$ such that for every $1<i\leq m$ and part $P\in \P_i$, $P$ comprises of at most $\Delta$ parts of $\P_{i-1}$.
%\end{definition}

%So informally speaking, in a merge sequence of valency $\Delta$ we can merge at most $\Delta$ parts together at a time.
 
\newcommand{\enc}{{\cal E}}
\newcommand{\dec}{{\cal D}}

\begin{definition}
	Let $\CC$ be a graph class. A {\em{labelling scheme}} for~$\CC$ is a pair of functions $(\enc,\dec)$, called the {\em{Encoder}} and the {\em{Decoder}}. The Encoder gets on input a graph $G\in \CC$ and a vertex $u$ of~$G$, and outputs a binary word $\enc_G(u)\in \{0,1\}^*$, called the {\em{label}} of $u$. The Decoder gets on input a pair of words $x,y\in \{0,1\}^*$ and outputs a single bit: the pair is either accepted or rejected. We require that the Decoder faithfully decodes the adjacency relation of graphs in~$\CC$ from the labels produced by the Encoder in the following sense: for every graph $G\in \CC$ and a pair of distinct vertices $u,v$ of~$G$, 
	\[uv\in E(G)\quad\textrm{if and only if}\quad \dec(\enc_G(u),\enc_G(v))=\textrm{accept}. \]
	The {\em{length}} of the labelling scheme $(\enc,\dec)$ is the function $\ell\colon \N\to \N$ mapping every $n\in \N$ to the maximum length of a label produced by the Encoder in a graph $G\in \CC$ with $n$ vertices:
	\[\ell(n)\coloneqq \max_{\substack{G\in \CC,\\|V(G)|=n}}\max_{u\in V(G)} |\enc_G(u)|.\]
	We may also say that a graph class $\CC$ admits a labelling scheme of length $\mathscr F$, where $\mathscr F$ is a class of functions (e.g. $\Oh(\log n)$), if there is a labelling scheme for $\CC$ of length $f\in \mathscr F$.
\end{definition}

Note that the definition of an adjacency labelling scheme does not stipulate any bounds on the time complexity of computing the Encoder and the Decoder, or even assert the computability of those functions. While these aspects can be of course discussed, we do not consider them~here.

With the definition in place, let us recall the result: the goal of this section is to prove the following.

\thmAdjLab*

Towards \cref{thm:adj-labelling}, we prove the following technical claim. 

\begin{theorem}\label{thm:adj-labelling-technical}
	Suppose that $\CC$ is a class of graphs such that every $n$-vertex graph $G\in \CC$ admits a merge sequence of length at most $m(n)$, valency at most $\Delta(n)$, and radius-$1$ width at most $k(n)$, for some functions $m,\Delta,k\colon \N\to \N$. Then $\CC$ admits a labelling scheme of length $\Oh(m(n)\cdot k(n)\cdot \log \Delta(n))$.
\end{theorem}

Note that \cref{thm:adj-labelling-technical} implies \cref{thm:adj-labelling}, for assuming $\CC$ has bounded short merge-width, we may take $m(n)\in \Oh(\log n)$, $\Delta(n)\in \Oh(1)$, and $k(n)\in \Oh(1)$. The remainder of this section is devoted to the proof of \cref{thm:adj-labelling-technical}.

\paragraph{Encoding thin trees.}
As the first technical step, we prove an auxiliary lemma about encoding thin trees. Throughout this section, for a parameter $\Delta\in \N$, by a {\em{$\Delta$-branching tree}} we mean a non-empty, finite set  of $T\subseteq [\Delta]^*$ (that is, consisting of finite words over the alphabet $[\Delta]=\{1,\ldots,\Delta\}$) that is closed under prefixes. We call the elements of $T$ {\em{nodes}} and interpret the prefix relation in $T$ as the ancestor relation: a node $u$ is an ancestor of a node $v$ if $u$ is a prefix of $v$. The usual notions of descendant, parent, and child then follow naturally. For $u\in T$ and $i\in [\Delta]$, the {\em{$i$-child}} of $u$ is the node~$ui$. Note that the set of values $i$ such that $u$ has an $i$-child in $T$ can be an arbitrary subset of $[\Delta]$. 

The {\em{depth}} of a node $u\in T$ is the length of $u$, and the {\em{height}} of~$T$ is the largest depth among the nodes of $T$. For $j\in \N$, by $T_j$ we denote the set of nodes of $T$ of depth exactly $j$.
A $\Delta$-branching tree $T$ is called {\em{$q$-thin}} if $|T_j|\leq q$ for every $j\in \N$.

\newcommand{\cd}{\mathsf{code}}

\begin{restatable}[name=\deferred]{lemma}{treeencoding}\label{cl:tree-encoding}
	With every $q$-thin $\Delta$-branching tree $T$ of height at most $h$ one can associate a binary word $\cd(T)\in \{0,1\}^*$ of length $\Oh(hq\log \Delta)$ so that $T$ can be uniquely computed from $\cd(T)$.
\end{restatable}
\begin{proof}
	For every $0\leq j\leq h$, let $u^1_j,\ldots,u^{q_j}_j$ be the elements of $T_j$ ordered lexicographically. Note that $q_j\leq q$ by $q$-thinness of $T$. For every $0\leq j\leq h$ and $t\in [q_j]$, consider the following two numbers:
	\begin{itemize}[nosep]
		\item $c^t_j\in \{0,1,\ldots,\Delta\}$ is the number of children that $u^t_j$ has in $T_{j+1}$, and
		\item $i^t_j\in [\Delta]$ is such that $u^t_j$ is the $i^t_j$-child of its parent. (This value is not recorded for $j=0$.)
	\end{itemize}
	Observe that given the sequences of numbers $(c^1_j,\ldots,c^{q_j}_j)$ and $(i^1_j,\ldots,i^{q_j}_j)$ for all $0\leq j\leq h$, we can uniquely reconstruct $T$ layer by layer. First, we have $T_0=\{\varepsilon\}$. Then, to construct $T_{j+1}$ from $T_j$, for each $t\in [q_j]$ we insert into $T_{j+1}$ the following $c^t_j$ children of $u^t_j$: the $i^{s+1}_{j+1}$-child, the $i^{s+2}_{j+1}$-child, and so on up to the $i^{s+c^t_j}_{j+1}$-child, where $s=\sum_{t'<t} c^{t'}_j$.
	
	It remains to note that since the sequences $(c^1_j,\ldots,c^{q_j}_j)$ and $(i^1_j,\ldots,i^{q_j}_j)$ consist of at most $q$ numbers in $\{0,1,\ldots,\Delta\}$, we can encode them in a binary word $\cd(T)$ of length $\Oh(hq\log \Delta)$.
\end{proof}

 We now proceed with the labelling scheme. 
\begin{proof}[Proof of \cref{thm:adj-labelling-technical}]
Let $(\P_1,R_1),\ldots,(\P_m,R_m)$ be a merge sequence of the given graph $G$ with length $m\leq m(n)$, valency $\Delta\leq \Delta(n)$, and radius-$1$ merge-width $k\leq k(n)$. Recall that $\P_1\preceq \P_2\preceq\ldots \preceq \P_m$, with $\P_1$ being the partition into singletons, and $\P_m$ the partition into one part:~$V(G)$. We may assume without loss of generality that $R_1=\emptyset$.

We construct a $\Delta$-branching tree $T$ of height $h\coloneqq m-1$ that encodes the sequence of partitions $\P_1\preceq \P_2\preceq\ldots \preceq \P_m$. Each level $T_j$, $0\leq j\leq h$, will correspond to the partition $\P_{m-j}$, and this correspondence is encoded by a bijection $\eta_j\colon T_j\to \P_{m-j}$. The construction proceeds as follows:
\begin{itemize}[nosep]
	\item For the unique part $V(G)\in \P_m$, we insert the empty word $\varepsilon$ to $T$ and set $\eta_0(\varepsilon)\coloneqq V(G)$.
	\item Next, for each consecutive $j\in [h]$, we construct $T_j$ and $\eta_j$ as follows. Consider every part $P\in \P_{m-j+1}$ and let $Q^1,\ldots,Q^d$ be the parts of $\P_{m-j}$ contained in $P$, enumerated arbitrarily. Note that $d\leq \Delta$. We insert into $T$ nodes $w1,w2,\ldots,wd$, where $w=\eta^{-1}_{j-1}(P)$, and set $\eta_j(wt)\coloneqq Q_t$ for each $t\in [d]$. 
\end{itemize}
Note that thus, the descendant relation in $T$ corresponds to the inclusion relation between the parts of the partitions $\P_1,\ldots,\P_m$: for $A\in \P_i$ and $A'\in \P_{i'}$ with $i\leq i'$, we have $A\subseteq A'$ if and only if $\eta^{-1}_{m-i}(A)$ is descendant of $\eta^{-1}_{m-i'}(A')$. We define $\eta\colon T\to \bigcup_{i\in [m]} \P_i$ as $\eta\coloneqq \bigcup_{0\leq j\leq h} \eta_j$.

Consider any vertex $u$ of $G$. We define the {\em{identifier}} of $u$ to be the unique node $w(u)\in T_h$ such that $\eta_h(w(u))=\{u\}$. Next, for a part $P\in \P_i$, we say that $P$ is {\em{touched}} by $u$ if $u\in P$ or there is $v\in P$ such that $uv\in R_{i+1}$. Then we define
\[T^u\coloneqq \{w\in T\mid \eta(w)\textrm{ is touched by }u\}\subseteq T.\]
We note the following.

\begin{claim}
	For each vertex $u$ of $G$, $T^u$ is a $k$-thin $\Delta$-branching tree.
\end{claim}
\begin{claimproof}
	To prove that $T^u$ is a $\Delta$-branching tree it suffices to argue that $T^u$ is closed under taking prefixes, but this follows directly from the observation: if $P\in \P_i$ is touched by $u$, then any $P'\in \P_{i'}$, $i'\geq i$, that contains $P$ is also touched by $u$, because $R_{i'+1}\supseteq R_{i+1}$. That $T^u$ is $k$-thin is implied by the bound on the radius-$1$ width of $(\P_1,R_1),\ldots,(\P_m,R_m)$: for each $i\in [m]$, $u$ touches only those parts of $\P_i$ that intersect the set $\Ball^1_{R_{i+1}}(u)$, of which there are at most $k$ many.
\end{claimproof}

Consider any node $w\in T^u$ of depth $m-i$ and let $P\coloneqq \eta(w)$; then $P\in \P_i$. Recall that $u$ is either adjacent to all $v\in P$ with $uv\notin R_i$ and $v\neq u$, or non-adjacent to all such $v$. We define $b(u,w)\in \{0,1\}$ to be $1$ if it is adjacent, and $0$ if it is non-adjacent (if there are no such vertices~$v$, $b(u,w)$ is chosen arbitrarily).

The next claim is the key component of our labelling scheme: it allows to decode the adjacency between two vertices $u$ and $v$ from the values $b(u,w)$ for $w$ ranging over the nodes of~$T^u$.

\begin{claim}\label{cl:decode}
	Let $u,v$ be distinct vertices of $G$ and let $w$ be the deepest node of $T^u$ that is an ancestor of $w(v)$ in $T$. Then
	\[uv\in E(G)\qquad\textrm{if and only if}\qquad b(u,w)=1.\]  
\end{claim}
\begin{claimproof}
	Let $P\coloneqq \eta(w)$ and let $i$ be such that the depth of $w$ is $m-i$; then $P\in \P_i$. By the definition of $b(u,w)$, it suffices to argue that $uv\notin R_i$. Suppose otherwise, that $uv\in R_i$. 
	
	Since we assumed that $R_1=\emptyset$, we must have $i>1$, or equivalently $w$ is not a leaf of $T$. Let then $w'$ be the unique child of $w$ such that $v\in P'\coloneqq \eta(w')$; note that $P'\in \P_{i-1}$. By the choice of $w$ as the deepest ancestor of $w(v)$ that belongs to $T^u$, we have $w'\notin T^u$, which in particular implies that $P'$ is a part of $\P_{i-1}$ not touched by~$u$. It follows that $u$ is not connected to any element of $P'$ by a pair belonging to $R_i$; this is a contradiction with $uv\in R_i$.
\end{claimproof}

We can finally define our adjacency labelling scheme $(\enc,\dec)$. For any vertex $u$ of $G$, we define its label $\enc_G(u)$ to consist of the following three components:
\begin{itemize}[nosep]
	\item the identifier $w(u)$, encoded as a binary word of length $h\cdot \lceil\log \Delta\rceil$;
	\item the word $\cd(T^u)$ provided by \cref{cl:tree-encoding}; and
	\item the binary word listing all the values $b(u,w)$ for $w\in T^u$, in the lexicographic order over~$T^u$. 
\end{itemize}
By \cref{cl:tree-encoding} and due to $|T^u|\leq (h+1)k$, we can easily encode all these components in a binary word $\enc_G(u)$ of length $\Oh(hk\log \Delta)$. Now given the labels $\enc_G(u)$ and $\enc_G(v)$ of distinct vertices $u$ and~$v$, we can determine whether $u$ and $v$ are adjacent as follows: 
\begin{itemize}[nosep]
	\item By \cref{cl:tree-encoding}, from $\cd(T^u)$ we may compute the set~$T^u$.
	\item From $T^u$ and the identifier $w(v)$, we may compute the deepest ancestor of $w(v)$ that belongs to $T^u$.
	\item From the binary word included in $\enc_G(u)$ we may read the value $b(u,w)$, which by \cref{cl:decode} determines whether $u$ and $v$ are adjacent or not.
\end{itemize}
This constitutes the Decoder $\dec$ and completes the proof of \cref{thm:adj-labelling-technical}.
\end{proof}

% \bibliographystyle{alphaurl}
% \bibliography{refs}
\printbibliography

\appendix
\crefalias{section}{appendix} % cleveref appendices are broken somehow?
\renewcommand{\mainonly}[1]{}

\section{Proof of Lemma \ref{lem:informal}}\label{apx:informal-proof}
Here, we explain how to obtain \cref{lem:informal} from the arguments of \cite[Section~3]{flip-separability}.
\informallemma*

Recall that $\overline{G}$ denotes the edge-complement of~$G$.
Let $\diam(G) \eqdef \max_{u,v \in V(G)} \dist_G(u,v)$ denote the diameter of a graph.
The proof of \cite[Lemma~8]{flip-separability} uses elementary results on diameter of graphs and their complements.
\begin{lemma}[{\cite[Lemma~9]{flip-separability}}]\label{lem:diam-complement}
  For any graph~$G$, either $\diam(G) \le 3$ or $\diam(\overline{G}) \le 3$.
\end{lemma}
A variant of \cref{lem:diam-complement} for bipartite graphs is also needed.
Here, for a bipartite graph $B = (X,Y,E)$ (where~$X$ and~$Y$ are the two sides of the bipartition),
we consider the bipartite complement $\overline{B} = (X,Y,\overline{E})$
where $\overline{E} \eqdef XY \setminus E$.
\begin{lemma}[{\cite[Lemma~10 and~11]{flip-separability}}]\label{lem:diam-bipartite-complement}
  For any bipartite graph $B = (X,Y,E)$, one of the following cases holds:
  \begin{enumerate}
    \item $\diam(B) \le 6$,
    \item $\diam(\overline{B}) \le 6$,
    \item $B$ is exactly the disjoint union of two bicliques, or
    \item there are vertices $x^+,x^- \in X$ such that~$x^-$ is isolated and $x^+$~is adjacent to all of~$Y$, or symmetrically swapping~$X$ and~$Y$.
  \end{enumerate}
\end{lemma}

The first statement we need to modify is \cite[Lemma~12]{flip-separability}:
\begin{lemma}[Variant of {\cite[Lemma~12]{flip-separability}}]\label{lem:bipartite-informal-lemma}
  Let $B = (X,Y,E)$ be a bipartite graph, and pick arbitrary vertices $x \in X$, $y \in Y$.
  Partition~$X$ according to the neighbourhood of~$y$,
  i.e.\ as $X = X^+ \uplus X^-$ with $X^+ = N(y)$ and $X^- = X \cap \overline{N}(y)$,
  and similarly $Y = Y^+ \uplus Y^-$ according to~$x$.
  Then there is a new bipartite graph $B' = (X,Y,E')$ obtained by some $\{X^+,X^-,Y^+,Y^-\}$-flip satisfying:
  \[ \text{for any $uv \in E'$,} \quad \dist_B(u,v) \le 6 \quad \text{and} \quad \dist_{\overline{B}}(u,v) \le 6. \]
\end{lemma}
\begin{proof}
  We consider the four possible cases of \cref{lem:diam-bipartite-complement}.
  \begin{enumerate}
    \item If~$\diam(B) \le 6$, we choose $B' = \overline{B}$ (which is indeed an $\{X^+,X^-,Y^+,Y^-\}$-flip).
      Then for any $uv \in E(B')$, we immediately have $\dist_{\overline{B}}(u,v) = 1$,
      and the assumption on the diameter gives $\dist_B(u,v) \le 6$.
    \item If~$\diam(\overline{B}) \le 6$, we instead choose $B' = B$, and the reasoning remains the same.
    \item If~$B$ is the disjoint union of two bicliques, then observe that regardless of the choice of~$x,y$,
      these two bicliques must be $(X^+,Y^+)$ and $(X^-,Y^-)$ (if we picked~$x,y$ in the same biclique),
      or alternatively $(X^+,Y^-)$ and $(X^-,Y^+)$ (if we picked~$x,y$ in different bicliques).
      Either way, these two bicliques can be removed by a $\{X^+,X^-,Y^+,Y^-\}$-flip,
      i.e.\ we can pick~$B'$ to be edgeless, making the last condition trivial.
    \item Finally, assume that there are $x^+,x^- \in X$ that are fully adjacent and non-adjacent to~$Y$ respectively.
      Note that necessarily $x^+ \in X^+$ and $x^- \in X^-$.
      Define~$B'$ by flipping between~$X^+$ and~$Y$, and let us check the desired property.
      Take an edge~$uv \in E(B')$ with $u \in X, v \in Y$.

      Assume first that~$u \in X^+$.
      Then by choice of~$B'$, $uv$ is also an edge in~$\overline{B}$, so $\dist_{\overline{B}}(u,v) = 1$.
      In~$B$ on the other hand, $uy$ is an edge (by definition of~$X^+$, since $u \in X^+$),
      and so are~$x^+y$ and~$x^+v$ (by choice of~$x^+$).
      Thus~$\dist_B(x,y) \le 3$.
      If~$u \in X^-$ instead, the situation is similar, swapping the roles of~$B$ and~$\overline{B}$, and using~$x^-$ instead of~$x^+$.

      The last case, swapping~$X$ and~$Y$, is symmetrical.\qedhere
  \end{enumerate}
\end{proof}

\begin{proof}[Proof of \cref{lem:informal}]
  Consider a graph~$G$, a partition~$\P$ of~$V(G)$, a transversal~$S$ of~$\P$, and define $\P' \eqdef \P \wedge S$.
  Let us construct a $\P'$-flip~$G'$ of~$G$.

  For any two parts~$X,Y \in \P'$, let us denote by $G[X,Y]$ be the subgraph of~$G$ consisting of all edges between~$X,Y$
  (with the corner case $G[X,X] \eqdef G[X]$ when $X=Y$).
  In the construction of~$G'$, we will consider each pair of parts~$X,Y$ independently,
  and transform~$G[X,Y]$ into~$G'[X,Y]$ by a $\P'$-flip to ensure that for any edge~$uv \in E(G'[X,Y])$,
  \begin{equation}\label{eq:local-metric-conversion}
    \dist_{G[X,Y]}(u,v) \le 6 \quad \text{and} \quad \dist_{\overline{G[X,Y]}}(u,v) \le 6. 
  \end{equation}
  We will ensure that the construction of~$G'[X,Y]$ ensuring~\eqref{eq:local-metric-conversion} is independent of the edges between any other pair of parts of~$\P$,
  so that the graphs~$G'[X,Y]$ for each pair~$X,Y$ can be recombined into~$G'$ without issue.
  \begin{enumerate}
    \item For each part $X \in \P$, by \cref{lem:diam-complement}, either~$G[X]$ or its complement has diameter at most~3.
      If $\diam(G[X]) \le 3$, we flip~$X$ with itself, i.e.\ we choose $G'[X] = \overline{G[X]}$.
      Then for any $uv \in G'[X]$, we directly have $\dist_{\overline{G[X]}}(u,v) = 1$,
      while $\dist_{G[X]}(u,v) \le 3$ holds by assumption on the diameter.
      In the case $\diam(\overline{G[X]}) \le 3$, we instead pick $G'[X] = G[X]$ and the reasoning remains similar.
    \item For each pair~$X,Y \in \P$ of distinct parts,
      let~$x,y$ be the unique vertices in $X \cap S$ and $Y \cap S$ respectively,
      and define $X^+,X^-,Y^+,Y^-$ as in \cref{lem:bipartite-informal-lemma}.

      Now \cref{lem:bipartite-informal-lemma} applied to~$G[X,Y]$ and~$x,y$
      gives some bipartite graph~$G'_{XY}$ between~$X$ and~$Y$,
      which is obtained as a $\{X^+,X^-,Y^+,Y^-\}$-flip of~$G[X,Y]$, and such that for any~$uv \in E(G'_{XY})$,
      the vertices~$u,v$ are at distance at most~$6$ in both~$G[X,Y]$ and $\overline{G[X,Y]}$.

      Note that the partition $\P' = \P \wedge S$ (restricted to~$X \cup Y$) refines $\{X^+,X^-,Y^+,Y^-\}$,
      so~$G'_{XY}$ is also a $\P'$-flip of~$G[X,Y]$.
      Thus we can choose $G'[X,Y] = G'_{XY}$.
  \end{enumerate}
  Thus we have constructed a $\P'$-flip~$G'$ of~$G$ satisfying~\eqref{eq:local-metric-conversion}.
  To conclude, consider any edge $uv \in E(G')$, say with $u \in X$ and $v \in Y$ (with possibly $X=Y$).
  We want to prove $\dist_{\P}(u,v) \le 6$, or in other words that for any $\P$-flip~$G''$ of~$G$,
  we have $\dist_{G''}(u,v) \le 6$.
  Observe that either~$G[X,Y]$ or~$\overline{G[X,Y]}$ must be a subgraph of~$G''$.
  By~\eqref{eq:local-metric-conversion}, $u,v$ are at distance at most~6 in both~$G[X,Y]$ and~$\overline{G[X,Y]}$, and thus also in~$G''$, proving the statement.
\end{proof}

\section{Sparse quotients, quasi-isometry, and neighbourhood covers: the bounded twin-width case}\label{sec:tww-app}

In this section we reprove the results from \cref{sec:quasi-iso} in the restricted case of graphs of bounded twin-width. The reason for this is that in this setting, the proofs are much simpler and yield better bounds for relevant parameters of the constructions, so this may be of interest for readers primarily interested in twin-width.

We need some additional definitions. A graph class $\CC$ is {\em{weakly sparse}} if there exists $t\in \N$ such that no member of $\CC$ contains the biclique $K_{t,t}$ as a subgraph. We say that $\CC$ has {\em{bounded sparse twin-width}} if $\CC$ has bounded twin-width and is weakly sparse. As proved in~\cite{tww2-journal}, classes of bounded sparse twin-width have bounded expansion (see also~\cite{DreierGJMR22,DreierGJMR22-correction}).

We now prove the following analogue of \cref{lem:quasi-iso}.

\begin{lemma}\label{lem:quasi-iso-tww}
	Let $G$ be a graph such that $\tww(G)\leq k$, for some $k\geq 1$. Then there exists a partition $\F$ of the vertex set of $G$ and a total order $<$ on $\F$ such that
	\begin{enumerate}
		\item $\scol^<_r(G/\F)\leq \frac{(k+1)^{r+1}-1}{k}$, for each $r\in \N$, and
		\item each part of $\F$ has weak diameter at most $2$ in $G$. 
	\end{enumerate}
\end{lemma}
\begin{proof}
	We use the construction proposed in \cite{tww3} for the purpose of proving that graphs of bounded twin-width are $\chi$-bounded. In fact, this construction was also the inspiration for our proof of \cref{lem:quasi-iso}, so what follows is a simpler variant of the reasoning from that proof.
	
	Let $\P_1\prec \P_2\prec \ldots \prec \P_n$ be a contraction sequence for $G$ such that for every $i\in [n]$, every part $A\in \P_i$ is non-homogeneous to at most $k$ parts $B\in \P_i$ with $B\neq A$. Recall that $\P_1\prec \P_2\prec \ldots \prec \P_n$ is a maximal chain of partitions of $V(G)$: $\P_1$ is the partition into singletons, $\P_n$ is the partition into one part $V(G)$, and each partition $\P_{i+1}$ is obtained from $\P_i$ by merging some two parts.
	
	Call a set $A\subseteq V(G)$ {\em{dominated}} if there exists a vertex $u$ of $G$ such that $A\subseteq \Ball^1_{G}(u)$. 
	Call a part $P$ {\em{maximally dominated}} with {\em{index}} $i$ if $P\in \P_i$, but for every $j>i$, the part of $\P_j$ containing $P$ is not dominated. Let $\F$ be the set of all maximally dominated parts across all the partitions $\P_1,\ldots,\P_n$. Note that since every vertex $u$ belongs to some dominated part (namely $\{u\}\in \P_1$) and any two sets in $\bigcup_{i\in [n]} \P_i$ are either disjoint or contained in one another, it follows that $\F$ is a partition of $V(G)$. Since every part of $\F$ is dominated, its weak diameter is at most $2$.
	So it remains to bound the strong colouring numbers of the quotient graph $G/\F$.
	
	 %Note that since each partition $\P_{i+1}$ differs from $\P_i$ only by merging two parts, for every $i\in [n]$ there are at most two parts with index $t$ in $\F$.
	
	With each $P\in \F$ we associate the {\em{index}} of $P$ defined as the number $i$ such that $P$ is maximally dominated with index $i$.
	Order the parts of $\F$ by reverse indices, breaking any ties arbitrarily. That is, introduce a total ordering $<$ on $\F$ so that for any parts $P,Q\in F$ with $P<Q$, say with indices $i$ and $j$, respectively, we necessarily have $i\geq j$.
	
	The key observation of \cite{tww3} is that $<$ defined in this way provides an ordering of the vertex set of $G/\F$ whose {\em{degeneracy}} is at most $k+1$: for every part $P\in \F$ there are at most $k+1$ parts $Q<P$ that are adjacent to $P$ in $G/\F$. We now argue that in fact, all the strong colouring numbers of $G/\F$ can be similarly bounded.
	This argument was in fact already proposed in \cite[Proof of Theorem~2]{DreierGJMR22}, but we present it here for completeness.
	
	\begin{claim}\label{cl:sreach-tww}
		For every $P\in \F$ and $r\in \N$, we have
		\[|\sreach^<_r(P)|\leq \frac{(k+1)^{r+1}-1}{k}.\]
	\end{claim}
	\begin{claimproof}
		Consider any $Q\in \sreach^<_r(P)$, $Q\neq P$, and let $P=A_0,A_1,\ldots,A_{p-1},A_p=Q$ be a path in $G/\F$ witnessing this membership, where $1\leq p\leq r$. Note that we have $Q<P<A_t$ for each $t\in [p-1]$. Let $i$ and $j$ be the indices of $P$ and $Q$, respectively; then~$j\geq i$.
		
		For each $t\in \{0,1,\ldots,p-1\}$, let $B_t$ be the unique part of $\P_i$ that contains $A_t$; such a part exists due to $P<A_t$. Note that $P=A_0=B_0$. Also, for every $t\in \{0,1,\ldots,p-2\}$, since $A_t$ and $A_{t+1}$ adjacent in~$G/\F$, it follows that $B_t$ and $B_{t+1}$ are equal or adjacent in $G/\P_i$.
		 Further, since $A_{p-1}$ and $A_p=Q$ are adjacent in $G/\F$, there must exist a part $B_p\in \P_i$ such that $B_p\subseteq Q$ and $B_{p-1}$ and $B_p$ are adjacent in $G/\P_i$. %It follows that $\dist_{G/\P_i}(A,B_p)\leq p$.

		We now argue the following: For each $t\in \{0,1,\ldots,p-1\}$, part $B_t$ has degree at most $k+1$ in  $G/\P_i$. This is clear when $i=n$, for then $B_t$ is the unique part of $\P_n$; so assume $i<n$. 
		Let $B_t'$ be the unique part of $\P_{i+1}$ that contains $B_t$. Since $B_t$ contains a maximally dominated part $A_t$, it follows that $B_t'$ is not dominated. Hence, in partition $\P_{i+1}$, $B_t'$ is non-homogeneous to at most $k$ other parts (by the assumptions on the contraction sequence), complete to no other part (for otherwise it would be dominated), and anti-complete to all the other parts; so $B_t'$ has degree at most $k$ in $G/\P_{i+1}$. Since $\P_{i+1}$ differs from $\P_i$ only by some two parts being merged, it follows that $B_t$ has degree at most $k+1$ in $G/\P_i$.
		
		We conclude that the part $B_p$ can be reached in $G/\P_i$ from $A$ by a path of length at most $p\leq r$ such that all parts on this path, except possibly for $B_p$ itself, have degrees at most $k+1$ in $G/\P_i$. The number of such parts $B_p$ is therefore bounded by $(k+1)+(k+1)^2+\ldots+(k+1)^r$, for every next vertex on the path can be chosen among at most $k+1$ possibilities. Since different sets $Q$ belonging to $\sreach^<_r(P)$ give rise to different sets $B_p$ (due to being disjoint), it follows that $|\sreach^<_r(P)|\leq 1+(k+1)+(k+1)^2+\ldots+(k+1)^r=\frac{(k+1)^{r+1}-1}{k}$. (The additional summand $1$ accounts for the part $P$ itself.)
	\end{claimproof}
	
	From \cref{cl:sreach-tww} we conclude that $\scol^<_r(G)\leq \frac{(k+1)^{r+1}-1}{k}$ for every~$r\in \N$, as claimed.
\end{proof}

We now prove the analogue of \cref{thm:quasi-iso-be}.

\begin{theorem}\label{thm:quasi-iso-betww}
	For every graph class $\CC$ of bounded twin-width there exists a graph class $\DD$ of bounded sparse twin-width such that every graph from $\CC$ is $(3,2)$-quasi-isometric to some graph from $\DD$.
\end{theorem}
\begin{proof}
	For every graph $G\in \CC$, let $\F_G$ be the partition of the vertex set of $G$ provided by \cref{lem:quasi-iso-tww}. Let \[\DD\coloneqq \{G/\F_G\colon G\in \CC\}.\] Since $\CC$ has bounded twin-width, from \cref{lem:quasi-iso-tww} it follows that the $r$-strong colouring numbers of $\DD$ are bounded by a function of $r$, hence $\DD$ has bounded expansion. In particular, $\DD$ is weakly sparse. Further, using the same argument as in \cref{sec:quasi-iso} we may easily see that mapping every vertex $u$ of $G$ to the part of $\F_G$ to which $u$ belongs yields a $(3,2)$-quasi-isometry from $G$ to $G/\F_G$.
	
	It remains to show that $\DD$ also has bounded twin-width. For this, we use a standard argument based on first-order transductions and the fact that every transduction of a class of binary structures of bounded twin-width again has bounded twin-width~\cite{tww1}; we assume the reader's familiarity with these tools. It is well known, see e.g.~\cite{tww1,tww4}, that if $G$ is a graph of twin-width $k$ and $\P_1\prec \P_2\prec \ldots \prec \P_n$ is a contraction sequence for $G$ witnessing this, then there exists a total order $<$ on $V(G)$ such that the ordered graph $(G,<)$, regarded as a binary structure, again has twin-width at most $k$, and every part $P\in \bigcup_{i\in [n]} \P_i$ is {\em{convex}} in $<$, that is, consist of an interval of elements that are consecutive in $<$. The construction presented in the proof of \cref{lem:quasi-iso-tww} produces a partition $\F_G$ such that every part of $\F_G$ is also a part of some $\P_i$, $i\in [n]$; hence, all the parts of $\F_G$ are convex in $<$. It now follows that the partition $\F_G$ can be transduced from the ordered graph $(G,<)$. Namely, it suffices to mark, using a unary predicate, the $<$-minimal element of each part of $\F_G$, and then to recover the parts of $\F_G$ as maximal intervals in $<$ that do not contain marked elements (besides the $<$-minimal one). Once partition $\F_G$ is transduced, we can easily transduce the quotient graph $G/\F_G$. This means that $G/\F_G$ can be transduced, using a fixed first-order transduction, from the ordered graph $(G,<)$ of twin-width at most $k$. By the results of~\cite{tww1} it now follows that the twin-width of $G/\F_G$ is bounded by a function of $k$, hence class $\DD$ has bounded twin-width.
\end{proof}

And here is the analogue of \cref{thm:nbhdCov}.

\begin{theorem}\label{thm:nbhdCov-tww}
	Every graph of twin-width $k$ admits a neighbourhood cover of radius $8$ and overlap $\Oh(k^2)$.
\end{theorem}
\begin{proof}
	Let $G$ be the graph in question. Let $\F$ be the partition of $V(G)$ and $<$ the ordering of $\F$ provided by \cref{lem:quasi-iso-tww}. Note that
	\begin{gather*}\scol^<_1(G/\F)\leq k+2=\Oh(k),\\ \scol^<_2(G/\F)\leq k+2+(k+1)^2=\Oh(k^2).\end{gather*}
	Observe that for each $P\in \F$, we have
	\[\wreach^<_2(P)=\sreach^<_2(P)\cup \bigcup_{Q\in \sreach^<_1(P)}\sreach^<_1(Q).\]
	Hence, we have
	\[\wcol^<_2(G/\F)\leq \scol^<_2(G/\F)+\scol^<_1(G/\F)^2\leq \Oh(k^2).\]
	From \cref{lem:NbdCovWcol} we conclude that $G/\F$ admits a neighbourhood cover~$\K$ of radius $2$ and overlap $\Oh(k^2)$. Since each part $P\in \F$ has weak diameter at most $2$, the same argument as in the proof of \cref{lem:quasi-iso} shows that $\{\bigcup C\colon C\in \K\}$ is a neighborhood cover of $G$ of radius $8$ and overlap $\Oh(k^2)$. 
\end{proof}

\end{document}